\documentclass[twoside,11pt]{article}

\usepackage[utf8]{inputenc}
\usepackage{amsmath,amsthm}
\usepackage{amssymb}
\usepackage{enumitem}
\setitemize{noitemsep,topsep=0pt,parsep=0pt,partopsep=0pt}
\setenumerate{noitemsep,topsep=-0.2em,parsep=0pt,partopsep=0pt,label=(\alph*)}
\usepackage{xcolor}
\allowdisplaybreaks
\usepackage[color=green!40]{todonotes}
\usepackage{psfrag}
\usepackage[colorlinks]{hyperref}
\usepackage{floatrow}
\usepackage{units}

\newtheorem{theorem}{Theorem}
\newtheorem{remark}[theorem]{Remark}

\newtheorem{proposition}[theorem]{Proposition}

\theoremstyle{definition}
\newtheorem{definition}[theorem]{Definition}

\numberwithin{equation}{section}
\numberwithin{figure}{section}
\numberwithin{theorem}{section}

\DeclareMathOperator*{\argmin}{arg\,min}
\DeclareMathOperator*{\essup}{ess\,sup}

\newcommand{\blau}{\color{black}}

\newcommand{\R}{\mathbb R}

\newcommand{\N}{\mathbb N}
\newcommand{\sph}{\mathbb S}
\newcommand{\edot}{\,\cdot\,}

\newcommand{\tfun}{\mathcal  T}
\newcommand{\rfun}{\mathcal R}
\newcommand{\ffun}{\mathcal F}
\newcommand{\gfun}{\mathcal G}

\newcommand{\Wo}{\mathbf W}
\newcommand{\Ho}{\mathbf H}
\newcommand{\Fo}{\mathbf T}
\newcommand{\Fqpat}{\mathbf F}

\newcommand{\Mo}{\mathbf M}
\newcommand{\Ao}{\mathbf A}
\newcommand{\Ko}{\mathbf K}
\newcommand{\Io}{\mathbf I}

\newcommand{\rmd}{\mathrm d}
\newcommand{\eps}{\epsilon}
\newcommand{\coloneqq}{:=}

\newcommand\abs[1]{\left\vert#1\right\vert}
\newcommand\sabs[1]{\lvert#1\rvert}
\newcommand\norm[1]{\left\Vert#1\right\Vert}
\newcommand\snorm[1]{\Vert#1\Vert}

\newcommand{\enorm}{\left\|\;\cdot\;\right\|}
\newcommand\set[1]{\left\{#1\right\}}

\newcommand\sphet[1]{\{#1\}}

\newcommand{\al}{\lambda}

\newcommand{\dom}{\mathcal D}

\newcommand{\diam}{\operatorname{diam}}
\newcommand{\dist}{\operatorname{dist}}
\newcommand{\prox}{\operatorname{prox}}

\DeclareMathOperator{\ran}{ran}

\newcommand{\kl}[1]{\left(#1\right)}

\newcommand{\skl}[1]{(#1)}
\newcommand\inner[2]{\left\langle#1,#2\right\rangle}

\newcommand{\Om}{\Omega}
\newcommand{\La}{\Lambda}

\newcommand{\sigdel}{h_\sigma}
\newcommand{\mudel}{h_\mu}
\newcommand{\Phiad}{\Phi^*}

\title{Single-stage reconstruction algorithm for quantitative\\photoacoustic tomography}

\author{Markus Haltmeier${}^\clubsuit$,
Lukas Neumann${}^\diamondsuit$
and Simon Rabanser${}^\clubsuit{}^\diamondsuit$}

\date{\normalsize
${}^\clubsuit$Department of Mathematics, University of Innsbruck\\[-0.1em]
\normalsize
Technikestra{\ss}e 13, A-6020 Innsbruck, Austria
\\[1em]
\normalsize
${}^\diamondsuit$Institute of Basic Sciences in Engineering Science, University of Innsbruck\\[-0.1em]
\normalsize
Technikestra{\ss}e 13, A-6020 Innsbruck, Austria
\\[1em]
\normalsize
E-mail: {\tt\{markus.haltmeier,lukas.neumann,simon.rabanser\}@uibk.ac.at}
}

\begin{document}

\maketitle

\begin{abstract}
The development of efficient and accurate image reconstruction algorithms is one of the cornerstones of computed tomography.
Existing  algorithms for quantitative photoacoustic tomography currently
operate in a two-stage procedure: First an inverse source problem for the acoustic wave
propagation is solved,  whereas in a second step the optical parameters are estimated from
the result of the first step.  Such an approach has several drawbacks.   In this paper we therefore
propose  the use of single-stage reconstruction algorithms for quantitative photoacoustic
tomography, where  the optical parameters are directly reconstructed from the observed acoustical data.
In that context we formulate the image reconstruction problem of quantitative photoacoustic tomography as a single nonlinear inverse problem
by coupling the radiative transfer equation with the acoustic wave equation. The inverse problem
is approached  by Tikhonov regularization with a convex penalty in combination with the proximal gradient iteration for minimizing the Tikhonov functional.
We present numerical results, where the  proposed single-stage algorithm shows an improved reconstruction quality at a similar computational cost.


\smallskip\noindent{\bf AMS classification numbers.}
44A12, 45Q05, 92C55.
\end{abstract}

\section{Introduction}
\label{sec:intro}

Photoacoustic tomography (PAT)  is a  recently developed
medical imaging paradigm that  combines the  high spatial resolution of ultrasound imaging with the high contrast of optical imaging \cite{Bea11,KruKopAisReiKruKis00,WanAna11,Wan09b,XuWan06}.
Suppose a semitransparent sample is illuminated with a short pulse of electromagnetic energy near the visible range.
Then parts of the optical energy will be absorbed inside the sample which causes a
rapid,  non-uniform increase of temperature. The increase of temperature yields a spatially varying thermoelastic expansion which in turn induces an acoustic pressure wave (see Figure~\ref{fig:pat}).
The induced acoustic pressure wave is measured outside of the  object of interest, and mathematical algorithms are used  to recover an image of the interior.

\begin{figure}[htb!]
\begin{center}
  \includegraphics[width=0.85\textwidth]{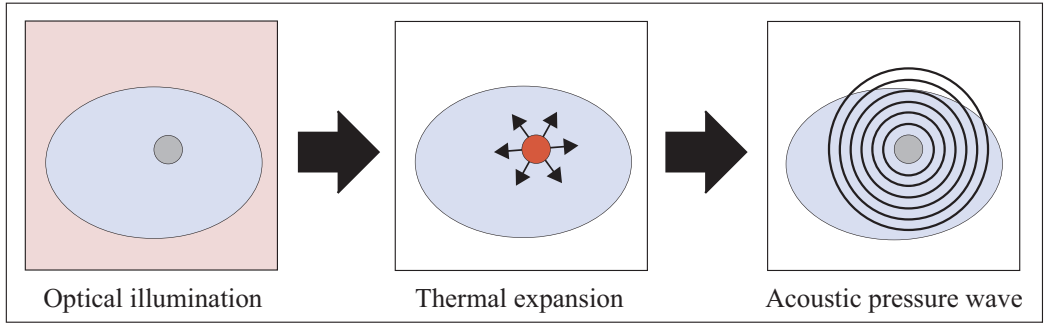}
\end{center}
\caption{\label{fig:pat}
\textsc{Basic principle of PAT.}  A semitransparent sample is illuminated with a  short optical pulse.
Due to optical absorption and subsequent thermal expansion within  the sample an acoustic pressure wave is induced. The acoustic pressure wave is measured outside of the sample and used to reconstruct an image of the interior.}
\end{figure}

Original (and also a lot of recent) work in PAT has been concentrated on the  problem of reconstructing  the initial pressure distribution, which has been considered as final image (see, for example, \cite{AgrKucKun09,BurBauGruHalPal07,FilKunSey14,FinHalRak07,FinPatRak04,FinRak07,Hal14,HalSchuSch05,KucKun08,Kow14,Kun07a,RosNtzRaz13,SteUhl09,XuWan06}). However, the recovered pressure distribution only provides indirect information about the investigated object.
This is due to the fact, that the initial pressure distribution is the product of the optical absorption coefficient
and the spatially varying optical intensity which again indirectly depends on the tissue parameters. As a consequence, the initial pressure distribution only provides qualitative information about the
tissue-relevant parameters.
Quantitative photoacoustic tomography (qPAT)  addresses exactly this issue and aims at quantitatively estimating the tissue parameters by supplementing the  wave inversion with an inverse problem for the light propagation in tissue (see, for example, \cite{AmmBosJugKang11,BalJolJug10,BalRen11,CheYa12,CoxArrBea07a,CoxArrKoeBea06,cox2012quantitative,DezDez13,KruLiuFanApp95,MamRen14,NatSch14,RenHaoZha13,RosRazNtz10,SarTarCoxArr,TarCoxKaiArr12,YaoSunHua10}).

To the best of our knowledge, {\blau apart from the very recent work \cite{Song:14},}
all existing  reconstruction algorithms for qPAT are currently performed via the following two-stage procedure:
First, the measured pressure values  are used to recover the initial pressure distribution caused by the thermal heating. In a second step, based on an  appropriate light propagation model, the spatially varying  tissue parameters are estimated from the initial pressure distribution recovered in the first step. However, any algorithm  for solving an inverse problem requires prior knowledge {\blau about} the  parameters to be recovered as well as partial knowledge  about the noise. If one solves qPAT  via a two-stage approach,
appropriate  prior information for the acoustic inverse problem is difficult to model, because the initial pressure depends on parameters not yet recovered.
This is particularly relevant for the case that the acoustic data can  only be measured on parts of the boundary (limited-angle scenario), in which case the acoustic inverse problems is known to be severely  ill-posed.  Further, using a two-stage approach, only limited information about the noise for the optical problem is available.

In view of such shortcomings  of  the two-stage approach,  in this paper we propose to recover
the optical parameters directly from the measured acoustical data via a single-stage procedure.
We work with the  stationary radiative transfer equation (RTE) as model for light propagation.
Our simulations show improved reconstruction quality of the proposed single-stage algorithm at a
computational cost similar to the one of {\blau existing  two-stage algorithms}.
Obviously our  single-stage strategy can alternatively  be combined with the diffusion approximation, which has also frequently been used in qPAT. {\blau In the present work} we use the stationary RTE since it is the more realistic model for light propagation in tissue.
{\blau In combination with the two-stage approach,
the RTE has previously been used for qPAT, for example, in \cite{BalJolJug10, DezDez13,TarCoxKaiArr12,SarTarCoxArr,MamRen14,YaoSunHua10}.}

\subsection{Mathematical modeling of qPAT}
\label{sec:pat-model}

Throughout this paper, let  $\Om \subset \R^d$ denote a convex bounded domain  with Lipschitz-boundary $\partial\Om$, where $d \in \set{2,3}$ denotes the spatial dimension.
We model the optical radiation by a function $\Phi \colon  \Om \times \sph^{d-1} \to \R$, where
 $\Phi \kl{x, \theta}$ is the density of photons at location $x \in \Om$ propagating in direction $\theta \in \sph^{d-1}$.
The photon density is supposed to satisfy the {\blau RTE, which reads}
\begin{multline}\label{eq:srt}
    \theta \cdot \nabla_x  \Phi\kl{x, \theta}
    + \kl{\sigma\kl{x} +\mu\kl{x}}  \Phi\kl{x, \theta}
    \\ =
    \sigma\kl{x} \int_{\mathbb \sph^{d-1}}  k \kl{\theta, \theta'} \Phi(x, \theta') \rmd \theta'
    + q(x,\theta)  \quad \text{ for } \kl{x,\theta} \in \Om \times \sph^{d-1} \,.
\end{multline}
Here $\sigma\kl{x} $ is the scattering coefficient,
$\mu\kl{x}$ is the absorption coefficient, and $q\kl{x, \theta}$ is the photon source density.
The scattering kernel $k \kl{\theta, \theta'}$ describes the redistribution of velocity directions of scattered photons
due to interaction with the background.
The stationary RTE \eqref{eq:srt} is commonly considered as a very accurate model for light transport in tissue (see, for example, \cite{Arr99,DaLiVol6,EggSch14c,Kan10}).

In order to obtain a well-posed problem one has to impose  appropriate boundary conditions. For that purpose it is convenient to split the boundary
$ \Gamma \coloneqq \partial \Om \times \sph^{d-1}$  into  inflow and  outflow boundaries,
\begin{align*}
\Gamma_-& \coloneqq
\left\{(x,\theta)\in\partial\Om\times \sph^{d-1}\colon \nu(x)\cdot\theta<0\right\} \,,
\\
\Gamma_+&\coloneqq
\left\{(x,\theta)\in\partial\Om\times
\sph^{d-1}
\colon \nu(x)\cdot\theta>0\right\} \,,
\end{align*}
with $\nu(x)$ denoting the outward pointing  unit normal at $x \in \partial \Om$.
We then augment \eqref{eq:srt} by  the inflow boundary conditions
\begin{equation}\label{eq:srt-b}
	\Phi|_{\Gamma_-}
	=  f
	\quad \text{ for some  }
	f \colon \Gamma_-  \to \R \,.
\end{equation}
Under physically reasonable assumptions it can be shown that the stationary RTE \eqref{eq:srt} together with the inflow boundary conditions \eqref{eq:srt-b} is a
well-posed problem. In Section~\ref{sec:transport} we apply a  recent result of  \cite{EggSch14} that guarantees the well-posedness of \eqref{eq:srt}, \eqref{eq:srt-b}
even in the presence of voids (parts of the domain under consideration, where $\mu$ and $\sigma$ vanish).

The absorption of photons causes a non-uniform heating of the tissue proportional
to the total amount of absorbed photons,
\begin{equation*}
   h \kl{x}
    \coloneqq
   \mu(x) \int_{\sph^{d-1}} \Phi(x, \theta) \rmd \theta
    \quad \text{ for } x \in \Om \,.
\end{equation*}
The heating in turn induces an
acoustic pressure wave $p \colon \R^d \times \kl{0, \infty} \to \R$.
The initial  pressure distribution is given by $p(\edot, 0) = {\blau \gamma} h$, where
${\blau \gamma}$ is  the Gr\"uneissen  parameter describing the efficiency of  conversion of heat into
acoustic pressure.
For the sake of simplicity we consider the Gr\"uneissen  parameter to be constant,
known and rescaled to one.
We further assume the speed of sound to be constant and also rescaled to one.
The photoacoustic pressure then satisfies the following initial value problem
for the standard wave equation,
\begin{equation}  \label{eq:wave-ini}
	\left\{ \begin{aligned}
	 \partial_t^2 p (x,t) - \Delta p(x,t)
	&=
	0 \,,
	 && \text{ for }
	\kl{x,t} \in
	\R^d \times \kl{0, \infty}
	\\
	p\kl{x,0}
	&=
	h(x) \,,
	&& \text{ for }
	x  \in \R^d
	\\
	\partial_t
	p\kl{x,0}
	&=0 \,,
	&& \text{ for }
	x  \in \R^d \,.
\end{aligned} \right.
\end{equation}
The goal of qPAT is to reconstruct the parameters $\mu$ and $\sigma$ from measurements of the acoustic pressure $p$ outside $\Omega$. {\blau Pressure measurements are usually taken  as a function of time on parts of the boundary $\partial \Omega$.}

\subsection{The inverse problem of qPAT}
\label{sec:pat-recon}

In  the following we assume that acoustic measurements are available for multiple optical source distributions  (illuminations). For that purpose, let  $\kl{q_i, f_i}$ for $i=1, \dots N$ be given pairs of source patterns
and {\blau boundary light sources}. We use  $\Fo_i$ to denote the operator that takes the  pair $\kl{ \mu, \sigma}$
to the solution of the stationary RTE  \eqref{eq:srt}, \eqref{eq:srt-b} with $q_i$ and $f_i$ in place of
$q$ and $f$, and denote by
\begin{equation*}%\label{eq:inip}
    \Ho_i \skl{ \mu, \sigma} \kl{x}
    \coloneqq
     \mu(x) \int_{\sph^{d-1}}
     \Fo_i \skl{ \mu, \sigma}(x, \theta) \rmd \theta
     \qquad \text{ for }
     x \in \Om  \,
\end{equation*}
the operator describing the corresponding thermal heating.
Further, we write $\Wo_{\Om, \La}$ for the operator that maps the  initial data
$h$ to the solution $\Wo_{\Om,\La} h  :=  p|_{\partial\Om \times (0, \infty) }$
of  the wave equation \eqref{eq:wave-ini} restricted to the boundary $\partial\Om$.
Appropriate functional analytic frameworks for $\Fo_i$, $\Ho_i$ and $\Wo_{\Om,\La}$
will be given in Section~\ref{sec:forward}, where we also study
 properties of these mapping.

The reconstruction problem of qPAT with multiple illuminations can be
written in the form of a nonlinear inverse problem,
\begin{equation} \label{eq:ip}
 	v_i = \kl{ \Wo_{\Om,\La} \circ \Ho_i } \skl{ \mu^\star, \sigma^\star} + z_i  \quad \text{ for } i = 1,\dots, N \,.
\end{equation}
Here  $v_i$ are the measured noisy  data, the operators $\Wo_{\Om,\La} \circ \Ho_i$ model the forward problem of qPAT, $z_i$ are the noise in the data, and $\mu^\star, \sigma^\star$ are the true parameters.
The aim is to estimate the parameter pair $\kl{ \mu^\star, \sigma^\star}$  from given data $v_i$, and hence solving the  inverse problem \eqref{eq:ip}.

\subsection{Outline of the paper}
\label{sec:outline}

In this paper we address the inverse problem \eqref{eq:ip} by Tikhonov regularization,
\begin{equation*}
	\frac{1}{2}
	 \sum_{i=1}^N
	 \norm{\kl{ \Wo_{\Om,\La} \circ \Ho_i }
	 \skl{ \mu, \sigma} - v_i}^2
	+ \lambda \rfun ( \mu, \sigma)
	\to \min_{(\mu, \sigma)} \,,
\end{equation*}
where $\rfun$ is a convex penalty and $\lambda > 0$ is the regularization parameter. We  show that Tikhonov regularization applied to single-stage qPAT
is well-posed and convergent; see Theorem~\ref{thm:tik}.
For that purpose we derive regularity results for the heating operators $\Ho_i$ in Section \ref{sec:forward}. To establish such properties we use  results for the stationary RTE
derived recently in  \cite{EggSch14}.

For numerically minimizing the Tikhonov functional we apply the proximal gradient algorithm (also named forward backward splitting); see Section~\ref{sec:min}. The  proximal gradient algorithm, is an iterative scheme for minimizing  functionals that can be written as the  sum of a smooth and a convex part \cite{ComWaj05,CombPes11}. For the classical two-stage approach in qPAT in combination with the diffusion approximation, the proximal gradient algorithm has recently been applied in
\cite{ZhaZhoZhaGao14}. Numerical results using the proximal gradient algorithm applied to our single-stage approach are presented in Section~\ref{sec:num},
where we also include a comparison with the two-stage approach.
Of course, our single-stage approach can be combined with classical gradient or Newton-type  schemes. The proximal gradient algorithm is our method of choice,
since its is very flexible and fast, and can be applied for a large class of smooth or non-smooth penalties.

\section{Analysis of the direct problem of qPAT}
\label{sec:forward}

Before actually studying  the inverse problem of qPAT we first make sure that the forward problem
is well-posed in suitable spaces and that the data depend continuously on the parameters we intend to reconstruct.
For that purpose we review a recent existence and uniqueness result for the stationary RTE allowing for voids
in the domain of interest \cite{EggSch14}.
The use of a-priori  estimates will lead to differentiability results for the operators $\Fo_i$ and $\Ho_i$.

\subsection{The stationary RTE}
\label{sec:transport}

The stationary RTE has been studied in various contexts.
The most prominent, apart from the transport of radiation in a scattering media, is reactor physics, where the equation is used in the group velocity approximation of the neutron transport problem.
An extensive collection of results regarding applications as well as existence and uniqueness of solutions can be found in \cite{DaLiVol6}.
The analysis of the RTE becomes considerably more involved if internal voids, i.e. regions where scattering and absorption coefficient become zero, are allowed.

Suppose $1 \leq p \leq \infty$,
and denote by $L^p\skl{\Gamma_-, \sabs{\nu \cdot \theta}}$ the space of all measurable functions $f$ defined on $\Gamma_-$  for which
\begin{equation*}
\norm{f}_{L^p\skl{\Gamma_-, \sabs{\nu \cdot \theta}}}
 \coloneqq
  \begin{cases}
  \sqrt[p]{\int_{\Gamma_-}   \abs{\nu(x) \cdot  \theta} \abs{f(x,\theta)}^p \rmd (x,\theta)}
  & \text{if $p < \infty$}
  \\
  \essup_{(x,\theta) \in \Gamma_-} \set{\abs{\nu(x) \cdot  \theta} \abs{f(x,\theta)} }
  & \text{if  $p  =  \infty$}
  \end{cases}
  \end{equation*}
is finite. We write $W^p(\Om\times \sph^{d-1})$ for the space of all measurable functions defined on
$\Om\times \sph^{d-1}$  such that
\begin{equation*}
\norm{\Phi}_{W^p\skl{\Om \times \sph^{d-1}}}^p
\coloneqq
\norm{\Phi}_{L^p\skl{\Om \times \sph^{d-1}}}^p+
\norm{\theta \cdot \nabla_x \Phi}_{L^p\skl{\Om \times \sph^{d-1}}}^p +
\norm{\Phi|_{\Gamma_-}}_{L^p\kl{\Gamma_-,
\abs{\nu \cdot \theta}}}^p
\end{equation*}
is well defined and finite (with the usual modification for $p=\infty$).
The subspace of
all $\Phi \in W^p(\Om\times \sph^{d-1})$  with $\Phi|_{\Gamma_-}=0$ will be denoted by
 $W^p_0(\Om\times \sph^{d-1})$.
Further, for a given scattering kernel $k \in L^\infty\skl{\sph^{d-1} \times \sph^{d-1}}$ we write
$\Ko \colon L^p\skl{\Om \times \sph^{d-1}}  \to L^p\skl{\Om \times \sph^{d-1}}$ for the corresponding scattering
operator,
\begin{equation*}
\kl{\Ko \Phi} \kl{x, \theta} =
\int_{\sph^{d-1}}  k(\theta, \theta' ) \Phi(x, \theta') \rmd \theta'  \qquad \text{ for }
(x, \theta) \in \Om \times \sph^{d-1} \,.
\end{equation*}
Throughout this article, the scattering kernel $k$  is supposed to be symmetric and nonnegative, and to satisfy
$\int_{\sph^{d-1}}  k \kl{\theta, \theta'} \rmd \theta' = 1$ for all $\theta \in \sph^{d-1}$. This reflects the fact that
$k\kl{\edot, \theta'}$ is a probability distribution describing the redistributions of velocity directions due to interaction of the photons  with the background. Under these assumption, the scattering  operator $\Ko$ is easily seen to be linear and bounded.

Using the notation just introduced, the stationary RTE  \eqref{eq:srt}, \eqref{eq:srt-b} can be written in the compact form
\begin{equation}\label{eq:strong}
\left\{
\begin{aligned}
\kl{ \theta \cdot \nabla_x
+ \kl{\mu   + \sigma - \sigma \Ko} }
\Phi
&=
q
&& \text{ in } \Om \times \sph^{d-1}\\
\Phi|_{\Gamma_-}
&=
f
&& \text{ on } \Gamma_-\,.
\end{aligned}
\right.
\end{equation}
By definition, a solution of the stationary RTE \eqref{eq:srt}, \eqref{eq:srt-b} in $W^p$ is any function $\Phi \in W^p\skl{\Om \times \sph^{d-1}}$ satisfying \eqref{eq:strong}.
The following theorem, which has been derived very recently in    \cite{EggSch14b}, states that  under physically reasonable  assumptions there exists exactly one such solution, that further continuously depends on  the {\blau source pattern} and the {\blau boundary light source}.

\begin{theorem}[Existence and uniqueness of solutions in $W^p$]  \label{thm:exist}
Let  $\overline\mu, \overline\sigma$ denote positive constants, let $\mu,\sigma$ be measurable functions satisfying  $0 \leq \mu \leq \overline\mu$ and $0 \leq \sigma \leq \overline\sigma$, and let $1\leq p\leq\infty$.
Then, for any source pattern $q\in L^p\skl{\Om\times \sph^{d-1}}$ and
 {\blau any boundary light source} $f \in L^p\kl{\Gamma_-, \abs{\nu \cdot \theta}}$, the stationary RTE \eqref{eq:srt}, \eqref{eq:srt-b} admits a unique solution  $\Phi\in W^p(\Om\times \sph^{d-1})$.
Moreover, there exists a
constant $C_p(\overline\mu, \overline\sigma)$ only depending on $p$, $\overline\mu$ and $\overline\sigma$, such that the following a-priori estimate holds
\begin{equation} \label{eq:apriori}
 \norm{\Phi}_{W^p\skl{\Om \times \sph^{d-1}}} 
\leq
 C_p(\overline\mu, \overline\sigma) \kl{    \norm{q}_{L^p\skl{\Om\times \sph^{d-1}}} + \norm{f}_{L^p\kl{\Gamma_-, \abs{\nu \cdot \theta}}} }\,.
\end{equation}
\end{theorem}

\begin{proof}
See \cite{EggSch14b}.
\end{proof}

\subsection{The parameter-to-solution operator $\Fo$ for the stationary RTE}
\label{sec:F}

Throughout this subsection, let $1 \leq p \leq \infty$,  and let $q\in L^\infty\skl{\Om\times \sph^{d-1}}$ and $f \in L^\infty\kl{\Gamma_-, \abs{\nu \cdot \theta}}$ be given source pattern and {\blau boundary light source}, respectively. Further, for fixed  positive numbers
$\overline\mu, \overline\sigma>0$ we denote
\begin{equation} \label{eq:dp}
\dom_p  \coloneqq \set{ (\mu, \sigma) \in  L^p \kl{\Om} \times L^p \skl{\Om \times \sph^{d-1} }:  0 \leq \sigma \leq  \overline\sigma
\text{ and } 0 \leq \mu \leq  \overline\mu }  \,.
\end{equation}
Then $\dom_p$ is a closed, bounded and convex subset of  $L^p \kl{\Om} \times L^p \kl{\Om \times \sph^{d-1} }$, that has empty interior in the case that $p < \infty$.

\begin{definition}[Parameter-to-solution operator for the stationary RTE]
The parameter-to-solution operator for the stationary RTE is defined by
\begin{align}\label{eq:operator}
&\Fo \colon \dom_p\to W^p(\Om \times \sph^{d-1}) \colon
\kl{\mu,\sigma}
\mapsto  \Phi \,,
\end{align}
where $\Phi$ denotes the unique solution of \eqref{eq:srt}, \eqref{eq:srt-b}.
\end{definition}

According to Theorem \ref{thm:exist} the operator  $\Fo$ is well defined.
Note further, that $\Fo$ depends  on  $p$, $q$, $f$, $\overline\mu$ and $\overline\sigma$.
Since these parameters will  be fixed in the following and in order to keep the notation simple we will not indicate the dependence of $\Fo$ on
these parameter explicitly.

Now we are in the position to state continuity  properties  of  $\Fo$ derived in  \cite{EggSch14c}.
We include a short proof of these results as its understanding is very useful  for the derivation of similar properties of the operator describing the heating that we investigate in the following subsection.

\begin{theorem}[Lipschitz continuity and weak continuity of $\Fo$]\label{thm:contF}\mbox{}\begin{enumerate}[topsep=0.5em,itemsep=0em,label=(\alph*)]
\item \label{it:contF-a}
The operator $\Fo$ is  Lipschitz-continuous.

\item \label{it:contF-b}
If $1 < p < \infty$, then $\Fo$ is sequentially weakly  continuous.
\end{enumerate}
\end{theorem}

\begin{proof} \mbox{}
\ref{it:contF-a}
Let $(\mu, \sigma), (\hat \mu, \hat \sigma) \in \dom_p$  be two given pairs of absorption and scattering coefficients and denote by  $\Phi \coloneqq \Fo (\mu, \sigma)$ and $\hat \Phi \coloneqq \Fo (\hat \mu, \hat \sigma)$ the corresponding solutions of  the stationary RTE.
Since  $q \in L^\infty \skl{\Om \times \sph^{d-1}}$ and
$f \in L^\infty\skl{\Gamma_-, \abs{\nu \cdot \theta}}$, Theorem~\ref{thm:exist} implies that  the difference $\hat \Phi - \Phi$ is an element of $W_0^\infty \skl{\Om \times \sph^{d-1}}$. Further, this  difference  is easily seen to satisfy
\begin{equation*}
\kl{\theta \cdot \nabla_x
+
\mu + \sigma - \sigma\Ko } \skl{\hat \Phi- \Phi}
 =
\kl{ \mu-\hat \mu }\hat \Phi + \skl{\sigma-\hat \sigma}
\hat \Phi -
\kl{\sigma-\hat\sigma} \Ko \hat \Phi\,.
\end{equation*}
Because $\Ko$ is a bounded linear operator on $L^p\skl{\Om \times \sph^{d-1}}$, the right hand side in the above  equation is actually contained in $L^p\skl{\Om \times \sph^{d-1}}$. Therefore, a further application of Theorem~\ref{thm:exist} yields
\begin{equation*}
 \snorm{\hat\Phi-\Phi}_{W^p \skl{\Om \times \sph^{d-1}}} \
 \leq C_p(\overline\mu, \overline\sigma) \snorm{\hat\Phi}_{L^\infty \skl{\Om \times \sph^{d-1}}}
 \kl{ \snorm{\mu-\hat\mu}_{L^p \skl{\Om}} + \norm{\Io - \Ko}_p  \snorm{\sigma-\hat\sigma}_{L^p \skl{\Om \times \sph^{d-1}}}}\,,
 \end{equation*}
where $\Io$ denotes the  identity and $\enorm{}_p$ the operator norm on $L^p \skl{\Om \times \sph^{d-1}}$. Since $\snorm{\hat\Phi}_{L^\infty \skl{\Om \times \sph^{d-1}}}$ is bounded independently of  $\hat\Phi$, this  implies the Lipschitz continuity of $\Fo$.

\ref{it:contF-b}
Let $(\mu_n, \sigma_n)_{n\in \N}$ be a sequence  in $\dom_p$ that converges weakly to the pair $(\mu, \sigma)\in \dom_p$, and denote by $\Phi_n = \Fo(\mu_n, \sigma_n)$ and  $\Phi = \Fo(\mu, \sigma)$ the corresponding solutions of the stationary RTE.
As in \ref{it:contF-a}, one argues  that  the difference $\Phi_n- \Phi $ is contained in $W_0^\infty \skl{\Om \times \sph^{d-1}}$ and satisfies
\begin{equation*}
\kl{\theta \cdot \nabla_x  + \mu + \sigma - \sigma\Ko } \kl{\Phi_n- \Phi}
 = \kl{ \mu- \mu_n }\Phi_n + \skl{\sigma- \sigma_n}
 \Phi_n -
\kl{\sigma-\sigma_n}  \Ko \Phi_n \,.
\end{equation*}
Now, from Theorem~\ref{thm:exist} it follows that $\theta \cdot \nabla_x  + \mu + \sigma - \sigma\Ko $ is invertible as an operator from $W_0^p \skl{\Om \times \sph^{d-1}}$ to  $L^p \skl{\Om \times \sph^{d-1}}$. Consequently, the inverse mapping   $\kl{\theta \cdot \nabla_x  + \mu + \sigma - \sigma\Ko}^{-1}$ is linear  and bounded and in particular weakly  continuous.
It therefore remains to show that  $ \kl{ \mu- \mu_n }\Phi_n + \skl{\sigma- \sigma_n}  \Phi_n -  \kl{\sigma-\sigma_n} \Ko \Phi_n$ weakly converges to zero in $L^p \skl{\Om \times \sph^{d-1}}$.
To see this,  denote by  $p_\ast=p/(p-1)$ the dual index and let  $\varphi \in L^{p_\ast}\skl{\Om \times \sph^{d-1}}$ be any element in the dual of $L^p\kl{\Om \times \sph^{d-1}}$. By Fubini's theorem we have
 \begin{equation*}
 \int_{\Om\times  \sph^{d-1}}\kl{ \mu(x)- \mu_n(x) }
 \Phi_n(x,\theta) \varphi(x,\theta) \, \rmd (x,\theta)
 =
  \int_{\Om} \kl{ \mu(x)- \mu_n(x) }
\kl{\int_{\sph^{d-1}} \Phi_n(x,\theta) \varphi(x,\theta)
\, \rmd \theta} \rmd x \,.
 \end{equation*}
The averaging lemma  (see, for example,  \cite{Mok97})
 implies that the averaging operator $\Ao \colon W^{p_\ast} \skl{\Om \times \sph^{d-1}} \to L^{p_\ast}\skl{\Om}\colon \Phi \mapsto  \int_{\sph^{d-1}} \Phi(\edot ,\theta)  \rmd \theta$  is compact for $1< p_\ast < \infty$.  Since $(\Phi_n)_{n\in \N}$ is bounded in $W^{\infty} \skl{\Om \times \sph^{d-1}} \subset W^{p_\ast} \skl{\Om \times \sph^{d-1}}$, this implies that  $\int_{\sph^{d-1}} \Phi_n(\edot,\theta) \varphi(\edot,\theta) \rmd \theta$ converges to $\int_{\sph^{d-1}} \Phi(\edot,\theta) \varphi(\edot,\theta) \rmd \theta$ with respect to $\enorm{}_{L^{p_\ast} \skl{\Om}}$. As $\mu_n\rightharpoonup \mu$  we can conclude that $ \skl{ \mu- \mu_n }\Phi_n$ converges to zero weakly.    In the same manner one shows $ \kl{ \sigma- \sigma_n }\Phi_n \rightharpoonup 0$. Finally, the equality
 \begin{equation*}
 \int_{\Om\times  \sph^{d-1}} \kl{ \mu- \mu_n }(x)
 \skl{\Ko \Phi_n}(x,\theta) \varphi(x,\theta) \rmd (x,\theta)
 =
  \int_{\Om} \kl{ \mu- \mu_n }(x)
\int_{\sph^{d-1}} \Phi_n(x,\theta) \skl{\Ko \varphi}(x,\theta) \rmd \theta \rmd x
 \end{equation*}
 and the use of similar  arguments
 show that
 $ \kl{ \sigma- \sigma_n } \Ko \Phi_n \rightharpoonup 0$.
\end{proof}

For the solution of the inverse problem of qPAT we will make use the  derivative of $\Fo$ that we compute next.
For that purpose we call $h \in  L^p(\Om) \times L^p(\Om \times \sph^{d-1})$ a feasible direction at $(\mu, \sigma) \in \dom_p$ if there exists some $\eps >0$ such that $(\mu, \sigma)  +  \eps h \in \dom_p$. Due to the convexity of $\dom_p$ we have $(\mu, \sigma)  +  s h \in \dom_p$ for all $ 0 \leq s \leq \eps$. The set of all feasible directions at $(\mu, \sigma) $ will be denoted by
$\dom_p\skl{\mu, \sigma} $. One immediately  sees that
\begin{equation*}
	\dom_p\skl{\mu, \sigma}
	= L^p(\Om) \times L^p(\Om \times \sph^{d-1})
	\quad \text{ if } 0 <  \mu < \overline\mu \text{ and }
	 0 <  \sigma < \overline\sigma \,.
\end{equation*}
For $(\mu, \sigma) \in \dom_p$ and any feasible direction $h\in \dom_p\skl{\mu, \sigma}$ we denote the  one-sided  directional derivative of $\Fo$ at $(\mu, \sigma)$ in direction $h$ by
\begin{equation} \label{eq:dd}
\Fo' (\mu, \sigma)(h) \coloneqq   \lim_{s \downarrow 0} \frac{\Fo((\mu, \sigma) + s h) -
\Fo (\mu, \sigma) }{s} \,,
\end{equation}
provided that the limit on the right hand side of \eqref{eq:dd} exists. If both  limits  $\Fo' (\mu, \sigma)(h)$ and $\Fo' (\mu, \sigma)(-h)$ exist and $h \mapsto \Fo' (\mu, \sigma)(h)$ is bounded and linear, we say that $\Fo$ is G\^ataux differentiable at $\skl{\mu, \sigma}$ and call $\Fo' (\mu, \sigma)$ the G\^ataux derivative of $\Fo$ at $(\mu, \sigma)$.

\begin{theorem}[Differentiability of $\Fo$]\label{thm:Fdiff}
For any $(\mu, \sigma) \in \dom_p$, the one-sided directional  derivative of $\Fo$ at  $(\mu, \sigma)$ in  direction $(h_\mu, h_\sigma)\in \dom_p\skl{\mu, \sigma}$ exists. Further, we have  $\Fo' (\mu, \sigma)(h_\mu, h_\sigma) = \Psi$, where  $\Psi$ is the unique solution of
\begin{equation}\label{eq:der}
\left\{
\begin{aligned}
\kl{ \theta \cdot \nabla_x
+
\kl{\mu + \sigma - \sigma\Ko} } \Psi
&=
- \kl{h_\mu + h_\sigma - h_\sigma\Ko} \Fo(\mu, \sigma)
&& \text{ in } \Om \times \sph^{d-1}\\
\Psi|_{\Gamma_-}
&=0
&& \text{ on } \Gamma_-\,.
\end{aligned}
\right.
\end{equation}
If $0 <  \mu < \overline\mu$ and
$0 < \sigma < \overline\sigma$, then $\Fo$ is G\^ateaux
differentiable at $\skl{\mu, \sigma}$.
\end{theorem}

\begin{proof}\mbox{}
 Suppose  $(\mu, \sigma)  \in \dom_p$ and let
 $h= (h_\mu, h_\sigma)\in \dom_p\skl{\mu, \sigma}$ be any feasible direction.
 For sufficiently small $s>0$  write
$\Phi_s \coloneqq  \Fo ( (\mu, \sigma) + sh)$ and $\Phi \coloneqq  \Fo (\mu, \sigma)$. As in the proof of Theorem \ref{thm:contF}  one shows that  $\Psi_s \coloneqq \skl{\Phi_s - \Phi}/s$ is contained in  $W^p_0 \skl{\Om \times \sph^{d-1}}$ and solves the equation $\skl{\theta \cdot \nabla_x
+
\mu + \sigma - \sigma\Ko }\Psi_s
 = -\skl{h_\mu + h_\sigma - h_\sigma\Ko }\Phi_s$.
 Consequently the difference
$\Psi_s- \Psi \in W^p_0 \skl{\Om \times \sph^{d-1}} $ solves
\begin{equation*}
\kl{\theta \cdot \nabla_x
+
\mu + \sigma - \sigma\Ko }\kl{\Psi_s- \Psi}
= - \skl{h_\mu + h_\sigma - h_\sigma\Ko } \skl{\Phi_s- \Phi} \,.
\end{equation*}
Application of the a-priori estimate of Theorem~\ref{thm:exist} shows the inequality  $
  \snorm{\Psi_s- \Psi}_{W^p \skl{\Om \times \sph^{d-1}}}
  \leq C_p(\overline\mu, \overline\sigma) \snorm{\Phi_s- \Phi}_{L^\infty \skl{\Om \times \sph^{d-1}}}
  ( \snorm{h_\mu }_{L^p \skl{\Om}} + \snorm{h_\sigma}_{L^p \skl{\Om \times \sph^{d-1}}})$
 Together with the continuity of $\Fo$   this implies
that the one-sided directional derivative $\Fo' \skl{\mu, \sigma }(h)$ exists and is given by $\lim_{s\to 0} \Psi_s = \Psi$. Finally, if $0 <  \mu < \overline\mu$ and
$0 < \sigma < \overline\sigma$, then  $h \mapsto \Fo' \skl{\mu, \sigma }(h)$ is bounded and linear and
therefore $\Fo$ is G\^ateaux differentiable at $\skl{\mu, \sigma}$.
\end{proof}

Note that  for any  parameter pair $\skl{\mu, \sigma} \in \dom_p$, the  solution
of \eqref{eq:der} depends linearly and continuously on $(h_\mu, h_\sigma) \in L^p(\Om)\times L^p(\Om\times \sph^{d-1})$. As a consequence, the one-sided directional derivative
can be extended to a bounded  linear operator
 \begin{equation} \label{eq:derF}
 	\Fo'(\mu, \sigma) \colon
	L^p(\Om)\times L^p(\Om\times \sph^{d-1}) \to W^p \skl{\Om \times \sph^{d-1}} \colon (h_\mu, h_\sigma)
	\mapsto \Psi \,,
\end{equation}
where $\Psi$ is the unique solution of  \eqref{eq:der}.
We refer to this extension as the derivative
of $\Fo$ at $(\mu, \sigma)$.

\subsection{The operator $\Ho$ describing the heating}
\label{sec:H}

Throughout this subsection,  let $q\in L^\infty\skl{\Om\times \sph^{d-1}}$ and $f \in L^\infty\kl{\Gamma_-, \abs{\nu \cdot \theta}}$ be given source pattern and {\blau boundary light source}, respectively.
As already mentioned in the introduction,
 photoacoustic signal generation due to the absorption of  light
 is described by the operator
\begin{equation*}
\Ho \colon  \dom_p
\to L^p(\Om) \colon
\kl{\mu,\sigma}
\mapsto  \mu\int_{\sph^{d-1}}
\Fo \skl{\mu, \sigma}(\edot,\theta)\rmd  \theta \,.
\end{equation*}
To shorten the notation, in the following we will make use of the averaging operator $\Ao  \colon W^p \skl{ \Om \times \sph^{d-1}} \to L^p \skl{ \Om}$ defined by $\Ao \Phi = \int_{\sph^{d-1}} \Phi\skl{\edot, \theta} \rmd \theta$.
By H\"olders inequality the averaging operator is well defined, linear  and bounded.
Using  the averaging operator we can write $\Ho \skl{\mu, \sigma} = \mu \Ao \Fo \skl{\mu, \sigma}$.

\begin{theorem}[Lipschitz continuity and weak continuity of $\Ho$] \label{thm:heating}\mbox{}
\begin{enumerate}[topsep=0.5em,itemsep=0em,label=(\alph*)]
\item \label{it:heating-a}
The operator $\Ho$ is  Lipschitz continuous.
\item \label{it:heating-b}
If $1< p < \infty$, then $\Ho$ is sequentially weakly continuous.
\end{enumerate}
\end{theorem}

\begin{proof}\mbox{} \ref{it:heating-a} Suppose that $(\mu, \sigma), (\hat \mu, \hat \sigma) \in \dom_p$ are two
pairs of admissible absorption and scattering coefficients.
The decomposition $\Ho \skl{\mu, \sigma} = \mu \Ao \Fo \skl{\mu, \sigma}$ and the triangle inequality imply
\begin{multline*}
 \norm{  \mu \Ao \Fo \skl{\mu, \sigma}
 - \hat \mu \Ao \Fo \skl{\hat\mu, \hat\sigma}}_{L^p(\Om)}
\\
\begin{aligned}
&=
\norm{  \mu \Ao \Fo \skl{\mu, \sigma}-  \hat\mu \Ao \Fo \skl{\mu, \sigma} + \hat\mu \Ao \Fo \skl{\mu, \sigma}
 - \hat \mu \Ao \Fo \skl{\hat\mu, \hat\sigma}}_{L^p(\Om)}
 \\
 &\leq
\norm{\Ao \Fo \skl{\mu, \sigma}}_{L^\infty(\Om)}
\norm{\mu-\hat\mu}_{L^p(\Om)} +
  \norm{\hat\mu}_{L^\infty(\Om)} \norm{\Ao \Fo \skl{\mu, \sigma} -  \Ao \Fo \skl{\hat\mu, \hat\sigma}}_{L^p(\Om)} \,.
\end{aligned}
\end{multline*}
According to Theorem \ref{thm:contF}, the operator $\Fo$ is Lipschitz continuous.
Because $\Ao$ is linear and bounded, also the composition $\Ao \Fo$ is  Lipschitz.
Noting that $\snorm{\Ao \Fo \skl{\mu, \sigma}}_{L^\infty(\Om)}$ and $\snorm{\hat\mu}_{L^\infty(\Om)}$ are  bounded by constants independent of  $\mu, \sigma$ and $\hat\mu, \hat\sigma$, this implies the  Lipschitz continuity of  $\Ho$.

\ref{it:heating-b}
Let $(\mu_n, \sigma_n)_{n\in \N}$ be a sequence in $\dom_p$ that converges weakly to $ ( \mu, \sigma) \in \dom_p$.
Since $\Fo$ is weakly continuous and $\Ao$ is linear
and bounded,   $(\Ao \Fo \skl{\mu_n, \sigma_n})_{n \in \N}$ converges weakly  to $\Ao \Fo \skl{\mu, \sigma}$.  Further, for  any  function $\varphi \in  L^{p_\star}(\Om)$, the dual space of $L^p(\Om)$, we have
\begin{multline*}
\abs{ \int_\Om  \kl{ \mu(x)  \skl{\Ao \Fo} \skl{\mu, \sigma}(x)
 -    \mu_n(x)  \skl{\Ao \Fo} \skl{\mu_n, \sigma_n}(x)    }\varphi(x) \,  \rmd x}
\\
\begin{aligned}
& \leq
\norm{\Ao \Fo \skl{\mu, \sigma} }_{L^\infty(\Om)}
\abs{ \int_\Om  \kl{ \mu(x)    - \mu_n(x)  }  \varphi(x)
\, \rmd x}
  \\
  &\qquad+
 \norm{\mu_n }_{L^\infty(\Om)}
\abs{ \int_\Om  \kl{  (\Ao \Fo) \skl{\mu, \sigma}(x)
 -    (\Ao \Fo) \skl{\mu_n, \sigma_n}(x)   } \varphi(x)
 \rmd x }\,.
\end{aligned}
\end{multline*}
The weak convergence of $\mu_n$ and $(\Ao \Fo \skl{\mu_n, \sigma_n})_{n \in \N}$  therefore implies the weak convergence of $ \mu_n \Ao \Fo \skl{\mu_n, \sigma_n}$ to $\mu \Ao \Fo \skl{\mu, \sigma}$ and shows the weak continuity of $\Ho$.
\end{proof}

Note that for the case $1 < p < \infty$, the  averaging operator $\Ao$ is  even  compact (see  \cite{Mok97}) which implies the
compactness of the composition $\Ao \Fo$.
As a consequence, for any given $\mu$, the partial mapping
$\sigma \mapsto \mu \skl{\Ao \Fo} (\mu, \sigma)$ is compact. It seems unlikely, however,  that the full operator
$\Ho$ is compact, too.

\begin{theorem}[Differentiability of $\Ho$]\label{thm:Hdiff}
For any $(\mu, \sigma) \in \dom_p$, the one-sided directional derivative of $\Ho$ at  $(\mu, \sigma)$ in
any feasible direction  $(h_\mu, h_\sigma) \in   \dom_p\skl{\mu, \sigma}$ exists. Further, we have
\begin{equation} \label{eq:derH}
\Ho' \skl{\mu,\sigma}(h_\mu, h_\sigma)
= h_\mu \int_{\sph^{d-1}}\Fo \skl{\mu, \sigma}(\edot,\theta)\rmd  \theta
+
\mu\int_{\sph^{d-1}} \Fo'\skl{\mu,\sigma}\skl{h_\mu, h_\sigma}(\edot,\theta)\rmd  \theta \,,
\end{equation}
where $\Fo'\skl{\mu,\sigma}\skl{h_\mu, h_\sigma}$ denotes the one-sides directional derivative
of $\Fo$ at  $(\mu, \sigma)$ in direction  $(h_\mu, h_\sigma)$ and can be computed as the solution of   \eqref{eq:der}.
Finally, if $0 <  \mu < \overline\mu$ and
$0 < \sigma < \overline\sigma$, then $\Ho$ is G\^ateaux differentiable at $\skl{\mu, \sigma}$.
\end{theorem}

\begin{proof}
Let $(\mu, \sigma)   \in \dom_p$ and   let
$(h_\mu, h_\sigma) \in   \dom_p\skl{\mu, \sigma}$ be a feasible direction.  For sufficiently small $s>0$,
we have
\begin{multline*}
\frac{\Ho \skl{(\mu, \sigma)+s (h_\mu, h_\sigma)} - \Ho\skl{\mu, \sigma} }{s}
\\
\begin{aligned}
&=
\frac{ (\mu + s h_\mu) \skl{\Ao \Fo}\kl{\skl{\mu, \sigma}+ s(h_\mu, h_\sigma)}  - \mu \skl{\Ao \Fo} \kl{\mu, \sigma} }{s}
\\
&=  h_\mu \skl{\Ao \Fo}\kl{\skl{\mu, \sigma}+ s(h_\mu, h_\sigma)}
+ \mu \frac{ \skl{\Ao \Fo}\kl{\skl{\mu, \sigma}+ s(h_\mu, h_\sigma)}  - \skl{\Ao \Fo} \kl{\mu, \sigma} }{s} \,.
\end{aligned}
\end{multline*}
According to Theorem~\ref{thm:contF}, the operator $\Fo$ is continuous and therefore the first term converges to $h_\mu \skl{\Ao \Fo} \skl{\mu, \sigma}$ as $s \to 0$.
Because $\Fo$ is  one-sided differentiable, see Theorem \ref{thm:Fdiff}, the second term converges to
$\mu  \Ao  \Fo'\skl{\mu,\sigma}\skl{h_\mu, h_\sigma}$.
Finally, if  $0 <  \mu < \overline\mu$ and  $0 < \sigma < \overline\sigma$, then $\Ho' \skl{\mu,\sigma}(h)$
is linear and bounded in the argument  $h$ which implies the G\^ateaux differentiability of $\Ho$ at $\skl{\mu, \sigma}$.
\end{proof}

Recall that for any $\skl{\mu, \sigma} \in \dom_p$, the derivative $\Fo'\skl{\mu,\sigma}$ is bounded and linear. Therefore,
the right hand side of \eqref{eq:derH}  depends linearly and continuously on $(h_\mu, h_\sigma) \in L^p(\Om)\times L^p(\Om\times \sph^{d-1})$.
As a consequence, the one-sided directional derivative of $\Ho$ at $\skl{\mu, \sigma} $  can be extended to a bounded
linear operator $\Ho'(\mu, \sigma) \colon L^p(\Om)\times L^p(\Om\times \sph^{d-1}) \to L^p \skl{\Om}$.
We will refer to this extension as the derivative of $\Ho$ at $(\mu, \sigma)$.  The derivative of $\Ho$
can be written in the form $\Ho'(\mu, \sigma)(h) = h_\mu  \Ao ( \Fo  (\mu, \sigma)) +
\mu \Ao  ( \Fo'(\mu, \sigma) (h))$.

\subsection{The wave operator $\Wo_{\Om,\La}$}
\label{sec:W}

Let $U \subset \R^d$ be a bounded and convex domain with smooth boundary.   We assume that $\bar \Om \subset U$ and write $L^2_\Om \skl{\R^d}$ for the space of all square integrable functions defined on $\R^d$ that are supported in $\bar \Om$. Likewise we denote by $C^\infty_\Om \skl{\R^d}$ the space of all infinitely differentiable functions defined on $\R^d$ having support in $\bar\Om$. Further, let $\La \subset  \partial U$ be a relatively open  subset of $\partial U$, and denote by $\diam(U)$  the maximal diameter of $U$ and
by $\dist(\Om, \La)$ the distance between $\Om$ and the observation surface $\La$.

\begin{definition}[The wave operator $\Wo_{\Om,\La}$] \label{def:W}
Let $w_{\Om,\La}  \colon  \kl{0, \infty} \to \R$ be a smooth, nonnegative, compactly supported function
with  $w_{\Om,\La} (t) = 1$ for all $\dist(\Om,\La) \leq t  \leq \diam(U) - \dist(\Om,\La)$.
We then define the wave operator  by
\begin{equation} \label{eq:solW}
\begin{aligned}
	\Wo_{\Om,\La} \colon C^\infty_\Om \skl{\R^d} \subset L^2_\Om \skl{\R^d}
	\to L^2 \kl{\La  \times\skl{0,\infty}}
	\colon h    \mapsto  w_{\Om,\La} \, p|_{\La\times\skl{0,\infty}}   \,,
\end{aligned}
\end{equation}
where $p$ denotes the unique solution of  \eqref{eq:wave-ini}.
\end{definition}

The operator $\Wo_{\Om,\La}$ maps the initial data of the wave equation \eqref{eq:wave-ini} to its solution restricted to  $\La \subset \partial U$ and models the
acoustic part of the forward problem of PAT. The cutoff function $w_{\Om,\La}$ accounts for the fact,
that in the two dimensional case the solution of the wave equation has unbounded
support in time but measurements can only be made over a finite time interval.

In the following we use a result from \cite{Pal10} to show  that
$\Wo_{\Om,\La}$ is a bounded linear and densely defined operator,
and therefore can be extended to a bounded linear operator on $L^2_\Om \skl{\R^d}$ in a unique manner.

\begin{theorem}[Continuity of the wave operator $\Wo_{\Om,\La}$]\label{thm:wave}
There exists some constant $c_{\Om,\La}$ such that
$
\snorm{\Wo_{\Om,\La} h}_{L^2\skl{\La\times\skl{0,\infty}}} \leq c_{\Om,\La} \snorm{h}_{L^2_\Om  \skl{\R^d}} $ for all $h \in C^\infty_\Om   \skl{\R^d}$.
Consequently, there exists a unique bounded  linear extension
\begin{equation*}
 \overline \Wo_{\Om,\La} \colon  L^2_\Om \skl{\R^d} \to L^2 \kl{\La\times\skl{0,\infty}} \quad \text{ with  } \;
 \overline \Wo_{\Om,\La}|_{C^\infty_\Om \skl{\R^d}} = \Wo_{\Om,\La} \,.
\end{equation*}
With some abuse of notation we again write $\Wo_{\Om,\La}$ for  $\overline \Wo_{\Om,\La}$ in the sequel.
\end{theorem}

\begin{proof}
It is sufficient to consider the case when the data are measured on the whole boundary
$\La  = \partial U$.
The well known explicit  formulas for the  solution of  \eqref{eq:wave-ini} in two and three spatial dimensions (see, for example, \cite{Eva98,Joh82}) imply that for every
$(y,t) \in \partial U  \times\skl{0,\infty}$ we have
\begin{equation}\label{eq:sm}
    \kl{\Wo_{\Om,\La} h}  \kl{y,t} =
     \begin{cases}
  \frac{ w_{\Om,\La}(t)}{2\pi}
  \, \partial_t
     \int_0^t \frac{r}{\sqrt{t^2-r^2}}\int_{\sph^{d-1}}  h \kl{y + r \omega}
     \rmd \omega \rmd r
     & \text{ for } d = 2
\\
      \frac{w_{\Om,\La}(t)}{4\pi}\,
      \partial_t \kl{t \int_{\sph^{d-1}} h \kl{y +  t \omega} \rmd \omega}
     & \text{ for } d = 3 \,.
     \end{cases}
\end{equation}
We define the  spherical mean Radon transform
$\Mo \colon C^\infty_\Om \skl{\R^d} \to C^\infty(\partial U \times\skl{0,\infty})$ by
$\Mo h \kl{y,t} \coloneqq  1/\omega_{d-1} \int_{\sph^{d-1}} h \skl{y +  t \omega} \rmd \omega
$ for $ (y,t) \in \partial U \times\skl{0,\infty} $.
With the spherical mean Radon transform, the wave operator can be written as
$\kl{\Wo_{\Om,\La} h}  \kl{y,t} =
     w_{\Om,\La}(t) \,    \partial_t
     \int_0^t r  \Mo h \kl{y,r}/\sqrt{t^2-r^2}
       \rmd r$ in the case of two spatial dimensions   and
  $  \kl{\Wo_{\Om,\La} h}  \kl{y,t} =w_{\Om,\La}(t) \,  \partial_t  \kl{t \Mo h}  \kl{y,t}$ in the three dimensional case.

Next we use a Sobolev estimate derived in  \cite{Pal10}, which states that for every $\al \in \R$ there exists a constant $c_{K, \lambda}$ such that    $\snorm{\Mo h}_{H^{\lambda+(d-1)/2}(\partial U \times\skl{0,\infty})}  \leq c_{K, \lambda}
\norm{h}_{H^\lambda(U)} $ for any $h \in C^\infty_\Om \skl{\R^d}$.
Application of this identity with $\al = 0$ and using the smoothing properties of the Abel transform by degree $1/2$ for the case of two spatial dimensions  yields the  continuity of $\Wo_{\Om,\La}$ with respect to the $L^2$ topologies. In particular, $\Wo_{\Om,\La}$ has a unique bounded linear extension to
$L^2_\Om \skl{\R^d}$.
\end{proof}

For solving the inverse problem of qPAT we
will further utilize an explicit expression for the adjoint of $\Wo_{\Om,\La}$,
that we compute next.

\begin{proposition}[Adjoint of the wave operator]
For $v \in L^2 \skl{\La\times\skl{0,\infty}} \cap  C^1 \skl{\La\times\skl{0,\infty}}$,
and every $x \in \bar \Om$, we have
\begin{equation} \label{eq:wave-ad}
    \kl{\Wo_{\Om,\La}^* v}\kl{x}
    =
      \begin{cases}
      \displaystyle
      -\frac{1}{2\pi}
    \int_{\La}
     \int_{\sabs{x-y}}^\infty
    \frac{\partial_t \skl{w_{\Om,\La} v} \kl{y,t}}
    {\sqrt{r^2 - \sabs{x-y}^2}}   \, \rmd r  \, \rmd S(y)
    & \text{ if } d = 2
    \\
    \displaystyle
    -\frac{1}{4\pi}
    \int_{\La} \frac{ \partial_t ( w_{\Om,\La} v)\kl{y,\abs{x-y}}}{\abs{x-y}} \, \rmd S(y)
    & \text{ if } d = 3
     \,.
\end{cases}
\end{equation}
\end{proposition}

\begin{proof}
This is a simple application of Fubini's theorem and the explicit expression for $\Wo_{\Om,\La} h$
given in~\eqref{eq:sm}.
\end{proof}

\section{Single-stage approach to qPAT}
\label{sec:inv}

In this section we solve the inverse problem of qPAT by a single-stage approach.
Our setting allows acoustic measurement for multiple sources. Such a strategy has been called multi-source qPAT or multiple illumination qPAT (see \cite{BalRen11,CoxTarArr11,Zem10}).
For that purpose, throughout this section  $q_i \in L^\infty\skl{\Om \times \sph^{d-1}}$ and  $f_i \in L^\infty\kl{\Gamma_- , \sabs{\nu \cdot \theta}}$, for $i = 1, \dots N$, denote given source patterns and {\blau boundary light sources},  respectively. Recall that  $\Om \subset \R^d$ denotes a  bounded
convex domain with Lipschitz boundary and $\Gamma_-$ denotes the inflow
boundary consisting of all pairs $(x, \theta) \in \partial \Om \times
\sph^{d-1}$ with $\nu(x) \cdot \theta  < 0$.

\psfrag{L}{$\La$}
\psfrag{U}{$U$}
\psfrag{O}{$\Om$}
\begin{figure}[htb!]
\floatbox[{\capbeside\thisfloatsetup{capbesideposition={right,bottom},capbesidewidth=0.5\textwidth}}]{figure}[\FBwidth]
{\caption{\textsc{Setup for single-stage qPAT.} The stationary RTE governs the light propagation
in the domain $\Om$. the absorption of photons induces an initial pressure wave proportional to the heating $\Ho_i(\mu, \sigma)$. Further, $\bar\Om$ is supposed to be contained in another domain $U$, and the pressure waves are measured with acoustic detectors located on a open subset $\La \subset \partial U$ of the boundary of $U$.}\label{fig:geometry}}
{ \includegraphics[width=0.35\textwidth]{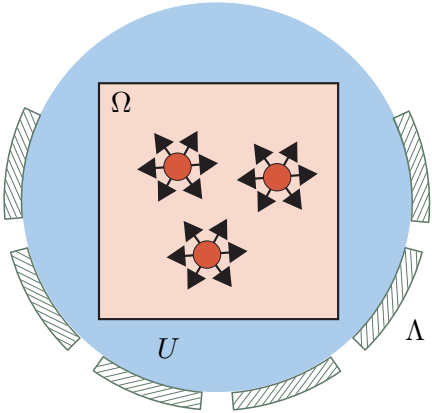}\hspace{0.1\textwidth}}
\end{figure}

To indicate the dependence of the solution of the stationary RTE
on the pair $(q_i, f_i)$ we write $\Fo_i \colon \dom_2 \to L^2\skl{\Om \times \sph^{d-1}}$
for the solution operator of the  stationary RTE \eqref{eq:srt}, \eqref{eq:srt-b}
with sources $(q_i, f_i)$. Here $\dom_2$ is the set of all
admissible pairs $(\mu,\sigma)$ defined in \eqref{eq:dp}.
Further,  we use
\begin{equation*}
\Ho_i \colon \dom_2
\to L^2(\Om) \colon \kl{\mu,\sigma}
 \mapsto  \mu\int_{\sph^{d-1}}
\Fo_i \skl{\mu, \sigma}(\edot,\theta)\rmd  \theta
\end{equation*}
to denote the corresponding operator describing the heating. For the coupling to the
acoustic problem,
it will be convenient to consider $\Ho_i(\mu, \sigma) \in L^2(\Om)$ as an element of $L^2_\Om(\R^d)$,
by extending it to a function defined on $\R^d$ that  is equal to zero on  $\R^d \setminus \Om$.

Further, recall the Definition \ref{def:W} of the wave operator $\Wo_{\Om,\La}$ modeling the acoustic problem, that maps the initial pressure $h$ in the wave equation \eqref{eq:wave-ini}  to its  solution
restricted  to  $\La \subset \partial U$. Here  $U$ is a convex domain with smooth boundary that contains the support of $f$.
To apply the results of Section~\ref{sec:forward}, in the following we assume that  $\bar \Om  \subset U$.
Then, according to the Section~\ref{sec:forward},  the operators $\Ho_i$ are Lipschitz continuous, weakly continuous and one-sided directional differentiable, and $\Wo_{\Om,\La}$ is  linear and bounded.
A practical representation of these domains is illustrated in Figure~\ref{fig:geometry}.

\subsection{Formulation as operator equation}
\label{sec:pat-equation}

In order to apply standard techniques for the solution of inverse problems
we  write the reconstruction problem of (multiple-source) qPAT as a single operator equation.
For that propose we denote by
\begin{equation*}
\begin{aligned}
 	\Fqpat \colon \dom_2
	& \to \kl{L^2 \skl{\La\times\skl{0,\infty}}}^N
	\\
	 (\mu, \sigma )
	 &\mapsto \kl{ \Wo_{\Om,\La} \circ \Ho_1 (\mu, \sigma ),
	\dots, \Wo_{\Om,\La} \circ \Ho_N (\mu, \sigma )}
\end{aligned}
\end{equation*}
the operator describing  the entire forward  problem of qPAT.
Further we denote by $\snorm{v}_N^2 \coloneqq \sum_{i=1}^N \snorm{v_i}^2_{L^2 \skl{\La\times\skl{0,\infty}}}$ the squared 
norm on  $L^2 \skl{\La\times\skl{0,\infty}}^N$.

\begin{theorem}[Properties of the forward operator of qPAT]\label{thm:prop}\mbox{}
\begin{enumerate}[topsep=0.5em,itemsep=0em,label=(\alph*)]
\item
The operator $\Fqpat$ is sequentially weakly continuous.
\item
The operator $\Fqpat$ is Lipschitz continuous.
\item
For any $(\mu, \sigma) \in \dom_2$, the one-sided directional derivative in any feasible  direction
$h \in \dom_2(\mu, \sigma) $ exists.
Further,
\begin{equation} \label{eq:derG}
\Fqpat' (\mu, \sigma)(h) =
\kl{ \Wo_{\Om,\La} \circ \Ho_1' (\mu, \sigma )(h),
\dots, \Wo_{\Om,\La} \circ \Ho_N' (\mu, \sigma )(h)} \,
\end{equation}
where $\Ho_i' (\mu, \sigma )(h)$ is given by \eqref{eq:derH} with $\Fo$ replaced by $\Fo_i$.
\item If $0 < \mu \leq \overline\mu$ and $0 < \sigma < \overline\sigma$, then $\Fqpat$ is G\^ateaux differentiable at $(\mu, \sigma)$.
\end{enumerate}
\end{theorem}

\begin{proof}
All claims follow from the corresponding properties of
the operators $\Ho_i $ (see Theorems~\ref{thm:heating} and~ \ref{thm:Hdiff})
and the boundedness of $\Wo_{\Om,\La}$ discussed in Theorem \ref{thm:wave}.
\end{proof}

The inverse problem of qPAT  with multiple illuminations
consists in solving the nonlinear equation
\begin{equation} \label{eq:ip3}
 	v   = \Fqpat \skl{ \mu^\star, \sigma^\star}
	+ z  \,,
\end{equation}
where $\skl{ \mu^\star, \sigma^\star}$ is the unknown, $v= (v_1 , \dots,  v_N )$
are the given noisy data, $\Fqpat \colon \dom_2 \to   L^2 \skl{\La\times\skl{0,\infty}}^N$  is the forward operator, and  $z$ is the noise in the data. Our single-stage approach for qPAT consists in
estimating the parameter pair $\skl{ \mu^\star, \sigma^\star}$ directly from \eqref{eq:ip3}.
In contrast, existing two-stage approaches for qPAT first construct estimates
$h_i$ for the heating functions $\Ho_i(\mu^\star, \sigma^\star) $ from data $v_i$ by
numerically inverting $\Wo_{\Om,\La}$, and subsequently solve
$(h_1, \dots h_N)  =  (\Ho_1 (\mu, \sigma ), \dots,  \Ho_N (\mu, \sigma ))$
for $(\mu, \sigma )$.
	
There are at least two common methods for tackling an inverse problem of  the form \eqref{eq:ip3}: Tikhonov type regularization methods  on the one
and iterative regularization methods on the other hand. In the following we apply Tikhonov regularization
to the inverse problem of qPAT.

\subsection{Tikhonov regularization for single-stage qPAT}
\label{sec:tikhonov}

We address the inverse problem \eqref{eq:ip3} by  Tikhonov regularization with general convex penalty.
For that purpose, let $
	\rfun \colon L^2 \kl{\Om} \times L^2 \kl{\Om\times \sph^{d-1}}
	\to \R \cup \set{\infty} $
be a convex, and lower semicontinuous functional with domain $\dom (\rfun) \coloneqq \sphet{(\mu, \sigma ) \in L^2 \kl{\Om} \times L^2 \kl{\Om\times \sph^{d-1}}\colon \rfun(\mu, \sigma )  < \infty}$.
We assume that $\rfun$ is chosen such that  $\dom_2 \cap \dom (\rfun)$ is  non-empty.

Tikhonov regularization  with penalty $\rfun$ consists  in
computing  a minimizer of the generalized Tikhonov functional
\begin{equation}\label{eq:tik}
\begin{aligned}
\tfun_{v,\lambda} \colon L^2 \kl{\Om} \times L^2 \skl{\Om\times \sph^{d-1}}
&\to \R \cup \set{\infty}
\\
(\mu, \sigma)
&\mapsto
	\begin{cases}
	\frac{1}{2}
	 \norm{\Fqpat( \mu, \sigma) - v}_N^2
	+ \lambda \rfun ( \mu, \sigma)
	& \text{ if }
	( \mu, \sigma) \in \dom_2 \cap \dom (\rfun)
	\\
	\infty & \text{ otherwise}\,.
	\end{cases}
\end{aligned}
\end{equation}
Here  $\lambda > 0$ is the so called regularization parameter which has to be
chosen accordingly, to balance between stability with respect to noise and accuracy in the case of exact data.
The data-fidelity term $\frac{1}{2} \snorm{\Fqpat( \mu, \sigma) - v}_N^2$
guarantees that any minimizer of \eqref{eq:tik} predicts the given data sufficiently well.
 The regularization term  $\lambda \rfun ( \mu, \sigma)$ on the other hand avoids
 over-fitting of the data and makes the reconstruction process well-defined and stable.

Note that Tikhonov regularization with penalty $\rfun$ is designed to stably
approximate a solution of the constrained optimization
problem
\begin{equation}\label{eq:Rmin}
\rfun (\mu, \sigma) \to \min_{(\mu, \sigma)\in \dom_2 \cap \dom (\rfun)}
\quad \text{ such that }
\Fqpat( \mu, \sigma) =  v^\star \,.
\end{equation}
Here $v^\star \in \ran(\Fqpat)$  is an element  in the
range of $\Fqpat$ and is referred to as exact data.
Any solution of \eqref{eq:Rmin} is called
$\rfun$-minimizing solution of the equation
$\Fqpat( \mu, \sigma) = v^\star$. Under the given assumptions there exists at
least one $\rfun$-minimizing solution, which however is not necessarily unique,
see \cite{SchGraGroHalLen09,SchKalHofKaz12}.

The properties of the operator $\Fqpat$ derived above and the use of general results from regularization
theory yield the  following result.

\begin{theorem}[Well-posedness and convergence of Tikhonov regularization]\label{thm:tik}\mbox{}
\begin{enumerate}[topsep=0.5em,itemsep=0em,label=(\alph*)]
\item
For data $v  \in L^2 \skl{\La\times\skl{0,\infty}}^N$  and every $\al >0$,
the Tikhonov functional  $\tfun_{v,\lambda}$
has at least one minimizer.

\item
Let $\al>0$, $v \in L^2 \skl{\La\times\skl{0,\infty}}^N$, and let $(v_n)_{n \in \N}$ be a sequence in $L^2 \skl{\La\times\skl{0,\infty}}^N$ with $\snorm{v - v_n}_N \to 0$.
Then  every  sequence of minimizers $(\mu_n, \sigma_n) \in \argmin \tfun_{v_n, \lambda}$  has a {\blau weakly} convergent subsequence.
Further, the limit $u$ of every weekly convergent subsequence
$(\mu_{\tau(n)}, \sigma_{\tau(n)})_{n\in \N}$  is a minimizer
$(\mu, \sigma)$ of $\tfun_{\lambda,v}$ and  satisfies ${\blau \rfun}(\mu_{\tau(n)}, \sigma_{\tau(n)}) \to \rfun(\mu, \sigma)$ for $n \to \infty$.

\item
Let $v^\star \in \ran(\Fqpat)$, let  $\kl{\delta_n}_{n\in \N} \subset (0, \infty)$ be a sequence converging to zero, and let $(v_n)_{n\in \N} \subset V$ be a sequence of data with  $\norm{v^\star - v_n}_N \leq \delta_n$.
Suppose further that $(\lambda_n)_{n\in \N} \subset (0, \infty)$
satisfies    $\al_n \to 0$ and $\delta_n^2/\al_n \to  0$ as $n \to \infty$.
Then the following hold:
\begin{itemize}[topsep=0em,itemsep=0em]
\item Every sequence $(\mu_n, \sigma_n) \in  \argmin \tfun_{v_k,\lambda_k}$ has a weakly converging subsequence.
\item
The limit  of every weakly  convergent subsequence
$(\mu_{\tau(n)},\sigma_{\tau(n)})_{n\in \N}$ of $(\mu_n, \sigma_n)_{n\in \N}$
is an $\rfun$-minimizing solution $(\mu^\star,\sigma^\star)$ of  $\Fqpat( \mu, \sigma) = v^\star$ and   satisfies $\rfun(\mu_{\tau(n)}, \sigma_{\tau(n)}) \to \rfun(\mu^\star,\sigma^\star)$.
\item If the $\rfun$-minimizing solution of $\Fqpat( \mu, \sigma) = v^\star$
is unique, then $(\mu_n, \sigma_n) \rightharpoonup (\mu^\star,\sigma^\star)$.
\end{itemize}
\end{enumerate}
\end{theorem}

\begin{proof}
 Since $\Fqpat$ is sequentially weakly  continuous (see Theorem~\ref{thm:prop}) and
 $\dom_2$ is closed and convex, this  follows from general results of Tikhonov regularization with convex penalties, see for example, \cite[Thm. 3.3, Thm. 3.4, Thm. 3.5]{SchGraGroHalLen09}.
\end{proof}

\subsection{Gradient of the data-fidelity term}
\label{sec:gradient}

For numerically minimizing the Tikhonov functional we will use the gradient
of the  data-fidelity term
\begin{equation} \label{eq:ffun}
 	\ffun(\mu, \sigma)
	\coloneqq
	\frac{1}{2}
	\norm{\Fqpat( \mu, \sigma) - v}_N^2
	=	
	\frac{1}{2}
	\sum_{i=1}^N
	\norm{\Wo_{\Om,\La} \mu \Ao \Fo_i(\mu,\sigma) - v_i}^2_{L^2\skl{\La \times\skl{0,\infty}}}
	\,.
\end{equation}
Recall that $v_i\in {\blau L^2\kl{\La \times\skl{0,\infty}} }$ are the given data,
$\Fo_i$ is the solution operator for the stationary RTE with source patterns $q_i$ and {\blau boundary light sources} $f_i$, $\Ao$ is the averaging operator, and $\Wo_{\Om,\La}$ is the solution operator for the wave equation.

Let $(\mu,\sigma) \in \dom_2$ be some admissible pair of parameters  and let
$(h_\mu,h_\sigma) \mapsto \ffun'(\mu,\sigma)(h_\mu,h_\sigma)$ denote  the one-sided directional derivative
of $\ffun$ at $(\mu,\sigma)$. We define the gradient  $\nabla \ffun (\mu,\sigma)$ of $\ffun$ at $(\mu,\sigma)$ to be any element in $L^2(\Om) \times L^2(\Om\times\mathbb{S}^{d-1})$ satisfying
\begin{equation} \label{eq:grad}
 \inner{\nabla \ffun(\mu,\sigma)}{(h_\mu,h_\sigma)}_{L^2(\Om) \times L^2(\Om\times\mathbb{S}^{d-1})}
 =
 \ffun'(\mu,\sigma)(h_\mu,h_\sigma)  \quad \text{ for }
 (h_\mu,h_\sigma) \in \dom_2(\mu, \sigma)
 \,.
\end{equation}
From Theorem~\ref{thm:prop} and the chain rule, it follows
that $\ffun'(\mu,\sigma)(h_\mu,h_\sigma)$ exists for any feasible direction
$(h_\mu,h_\sigma) \in \dom_2(\mu, \sigma)$.  Further, in the case that
$\mu$ and $\sigma$  are strictly positive, we have $\dom_2(\mu, \sigma) = L^2(\Om) \times L^2(\Om \times \sph^{d-1})$, which implies that
$\nabla \ffun (\mu,\sigma)$ is uniquely defined by \eqref{eq:grad}.

In order to  compute the gradient  we  derive a more explicit expression for
the one-sided directional derivative.

\begin{proposition}[One-sided directional derivative of the  data-fidelity  term]\label{prop:gradient}
Let  $(\mu, \sigma) \in \dom_2$ be an admissible pair of parameters and let
$(h_\mu, h_\sigma) \in   \dom_2\skl{\mu, \sigma}$ be a feasible direction.
Then we have
\begin{multline} \label{eq:gradient}
 \ffun'(\mu,\sigma)(h_\mu,h_\sigma)
 = \sum_{i=1}^N
 \inner{\Ao\Phi_i \Wo_{\Om,\La}^* \left[\Wo_{\Om,\La} \mu \Ao \Fo_i(\mu,\sigma) - v_i\right]- \Ao(\Phi_i\Phiad_i)}
 {\mudel}_{L^2(\Om)}
  \\
  +
 \sum_{i=1}^N \inner{- \Phi_i \Phiad_i + (\Ko\Phi_i) \Phiad_i}
 {\sigdel}_{L^2(\Om\times\mathbb{S}^{d-1})} \,,
\end{multline}
where $\Phi_i := \Fo_i (\mu, \sigma)$,  and  $\Phiad_i$ is  the unique solution of the adjoint equation
\begin{equation}\label{eq:phidag}
\kl{ -\theta \cdot \nabla_x
+
\kl{\mu + \sigma - \sigma\Ko} } \Phiad_i
= \Ao^\ast\mu\Wo_{\Om,\La}^* \left[\Wo_{\Om,\La} \mu \Ao \Phi_i - v_i\right]
\quad\text{ in } \Om \times \sph^{d-1}
\end{equation}
satisfying the zero outflow boundary condition $\Phiad|_{\Gamma_+} = 0$.
\end{proposition}

\begin{proof}
Obviously it is sufficient to consider the case $N=1$, where we write
$v$, $\Fo$, $\Phi$ and  $\Phiad$  in place of $v_i$, $\Fo_i$, $\Phi_i$ and $\Phiad_i$.
By  \eqref{eq:ffun} we have
\begin{align}
 &\nonumber\ffun'(\mu,\sigma)(h_\mu,h_\sigma)
 \\
 &\nonumber\hspace{1cm}
 =\inner{\Wo_{\Om,\La} \mu\Ao\Fo(\mu,\sigma)-v}{\Wo_{\Om,\La} \mudel\Ao\Fo(\mu,\sigma)+\Wo_{\Om,\La} \mu\Ao\Fo'(\mu,\sigma)(\mudel,\sigdel)}_{L^2\skl{\La \times\skl{0,\infty}}}
 \\
 & \nonumber\hspace{1cm}
 =\inner{\Wo_{\Om,\La}  \mu\Ao\Phi-v}{\Wo_{\Om,\La} \mudel\Ao\Phi}_{L^2\skl{\La \times\skl{0,\infty}}}
 \\
 & \nonumber\hspace{3cm}+ \inner{\Wo_{\Om,\La} \mu\Ao\Phi-v}{\Wo_{\Om,\La} \mu\Ao\Fo'(\mu,\sigma)(\mudel,\sigdel)}_{L^2\skl{\La \times\skl{0,\infty}}}
 \\
 & \nonumber\hspace{1cm}
 = \inner{ \Ao\Phi \Wo_{\Om,\La}^* \left[\Wo_{\Om,\La}  \mu\Ao\Phi-v\right]}
 {\mudel}_{L^2(\Om)}
 \\
 & \label{eq:term12}\hspace{3cm}+
 \inner{\Ao^\ast\mu\Wo_{\Om,\La}^*\left[\Wo_{\Om,\La}  \mu\Ao\Phi-v\right]}
 {\Fo'(\mu,\sigma)(\mudel,\sigdel)}_{L^2(\Om \times \sph^{d-1})}\,.
\end{align}

Recall  that  $\Phiad$ is the solution of~\eqref{eq:phidag} with source term
$q = \Ao^\ast\mu\Wo_{\Om,\La}^* \left[\Wo_{\Om,\La} \mu \Ao \Phi - v\right]$
and  zero outflow boundary conditions
$\Phiad|_{\Gamma_+} = 0$.
Further, according to Theorem~\ref{thm:Fdiff}, the one-sided directional derivative of
$\Fo$ is  given by  $\Fo' (\mu, \sigma)(h_\mu, h_\sigma) = \Psi$, where
$\Psi$ is the unique solution of  $\kl{ \theta \cdot \nabla_x
+ (\mu + \sigma - \sigma\Ko} ) \Psi =
- (h_\mu + h_\sigma - h_\sigma\Ko ) \Phi$ with inflow boundary conditions
$\Psi|_{\Gamma_-} = 0$. The zero outflow and zero  inflow boundary conditions of $\Phiad$
and $\Psi$, respectively, and one integration
by parts, show $\langle - \theta \cdot  \nabla_x \Phiad,  \Psi \rangle_{L^2(\Om \times \sph^{d-1})}=\langle \Phiad, \theta \cdot  \nabla_x \Psi \rangle_{L^2(\Om \times \sph^{d-1})}$.
Further, notice that the averaging operator $\Ao$ has adjoint $\Ao^* \colon   L^2(\Om)\rightarrow L^2(\Om \times \mathbb{S}^{d-1})$, $(\Ao^* g)(x,v) = g(x)$, and that  the scattering  operator $\Ko$ is self-adjoint.

Using these considerations, the second term  in \eqref{eq:term12} can be written
as
\begin{multline*}
 \inner{ \Ao^\ast\mu\Wo_{\Om,\La}^* \left[\Wo_{\Om,\La} \mu\Ao \Phi-v\right]}{\Fo'(\mu,\sigma)(\mudel,\sigdel)}_{L^2(\Om\times\mathbb{S}^{d-1})}\\
\begin{aligned}
&=
 \inner{
 \kl{ -\theta \cdot \nabla_x + \kl{\mu + \sigma - \sigma\Ko} } \Phiad}
 {\Psi}_{L^2(\Om\times\mathbb{S}^{d-1})}
 \\&=
 \inner{\Phiad}
 {\kl{ \theta \cdot \nabla_x + \kl{\mu + \sigma - \sigma\Ko} } \Psi}_{L^2(\Om\times\mathbb{S}^{d-1})}
 \\&=
 \inner{\Phiad}
 {-\left(\mudel+\sigdel-\sigdel \Ko\right)\Phi}_{L^2(\Om\times\mathbb{S}^{d-1})}
 \\&=
 -\inner{\Phi\Phiad}
 {\mudel+\sigdel}_{L^2(\Om\times\mathbb{S}^{d-1})}
 +\inner{(\Ko\Phi) \Phiad}
 {\sigdel}_{L^2(\Om\times\mathbb{S}^{d-1})}
 \\&=
   -\inner{\Phi\Phiad}
   {\Ao^\ast(\mudel)+\sigdel}_{L^2(\Om\times\mathbb{S}^{d-1})}
   +\inner{(\Ko\Phi) \Phiad}
   { \sigdel}_{L^2(\Om\times\mathbb{S}^{d-1})}
  \\&=
   \inner{ -\Ao(\Phi\Phiad)}
    {\mudel}_{L^2(\Om)}
    + \inner{-\Phi\Phiad + (\Ko \Phi) \Phiad}
    {\sigdel}_{L^2(\Om\times\mathbb{S}^{d-1})}
   \,.
 \end{aligned}
 \end{multline*}
Together with \eqref{eq:term12} this yields the desired identity
\eqref{eq:gradient}.
\end{proof}

Let  $(\mu, \sigma) \in \dom_2$ be an admissible pair of absorption and scattering coefficient.
If  $\mu, \sigma$ are both strictly positive, then one concludes from Proposition~\ref{prop:gradient} that
the gradient  of  $\ffun$ at $(\mu, \sigma)$ is uniquely defined and given by $\nabla \ffun(\mu,\sigma) = (\nabla_\mu \ffun(\mu,\sigma),  \nabla_\sigma \ffun(\mu,\sigma)) $ with
\begin{align} \label{eq:gradient2}
\nabla_\mu \ffun(\mu,\sigma)
&=
 \sum_{i=1}^N
 \kl{ \Ao\Phi_i \Wo_{\Om,\La}^* \left[\Wo_{\Om,\La} \mu \Ao\Phi_i - v_i\right]- \Ao(\Phi_i\Phiad_i)}
\\ \label{eq:gradient3}
\nabla_\sigma \ffun(\mu,\sigma)
&=
 \sum_{i=1}^N
 \kl{- \Phi_i \Phiad_i + (\Ko\Phi_i) \Phiad_i} \,.
\end{align}
Here $\Phi_i := \Fo_i (\mu, \sigma)$,  and  $\Phiad_i$ is  the solution of the adjoint equation
\eqref{eq:phidag} with zero outflow boundary condition.
In the case that $\mu, \sigma$ are not both strictly positive, the gradient is not uniquely defined by \eqref{eq:grad}. However, Proposition~\ref{prop:gradient} implies that the vector
$\nabla \ffun(\mu,\sigma)$ defined by \eqref{eq:gradient2}, \eqref{eq:gradient3} still satisfies \eqref{eq:grad}. We therefore take \eqref{eq:gradient2}, \eqref{eq:gradient3} as gradient
of $\ffun$ at any $(\mu, \sigma) \in \dom_2$.

\subsection{Proximal gradient algorithm for single-stage  qPAT}
\label{sec:min}

In order to  minimize the Tikhonov functional we apply the  proximal gradient (or forward backward splitting)  algorithm, which is an iterative algorithm for minimizing  functionals that can be written as a sum $\ffun+\gfun$, where $\ffun$ is smooth and  $\gfun$ is  convex \cite{ComWaj05,CombPes11}.
The proximal gradient algorithm  computes a sequence of iterates by alternating application of explicit gradient steps for the first functional $\ffun$ and implicit  proximal steps for  the second functional $\gfun$.

To apply the proximal gradient algorithm for minimizing the Tikhonov functional \eqref{eq:tik}  we take
$ \ffun(\mu, \sigma) = \frac{1}{2}\snorm{\Fqpat( \mu, \sigma) - v}_N^2$ for the first and
$\gfun(\mu, \sigma)  = \lambda \rfun ( \mu, \sigma)$ for the second functional.
The proximal gradient algorithm then generates a sequence $\skl{\mu_n, \sigma_n}$ of iterates
defined by
\begin{equation} \label{eq:proxgrad}
\skl{\mu_{n+1}, \sigma_{n+1}}
\coloneqq
\prox_{s_n\lambda\rfun}  \kl{ \kl{\mu_n, \sigma_n}
-  s_n  \nabla \ffun \kl{\mu_n, \sigma_n}} \quad \text{ for } n \in \N  \,.
\end{equation}
Here $\skl{\mu_0, \sigma_0} \in \dom_2 \cap \dom (\rfun)$  is an initial  guess,
$s_n>0$ is the step size in the $n$-th iteration,
 $\nabla \ffun(\mu,\sigma) = (\nabla_\mu \ffun(\mu,\sigma),  \nabla_\sigma \ffun(\mu,\sigma)) $
 is the gradient of $\ffun$ given by \eqref{eq:gradient2}, \eqref{eq:gradient3},
 and
\begin{equation} \label{eq:prox}
\prox_{s_n\lambda\rfun} \kl{\hat \mu, \hat \sigma }
\coloneqq
 \argmin_{( \sigma, \mu ) \in \dom_2 \cap \dom (\rfun)}
 \frac{1}{2} \norm{\kl{\mu, \sigma} - \kl{\hat \mu, \hat \sigma}}^2
+ s_n \lambda\rfun\skl{\mu, \sigma}
\end{equation}
is the proximity operator corresponding to the functional $s_n\lambda\rfun\skl{\mu, \sigma}$.

\begin{remark}[Lipschitz continuity of $\nabla \ffun$]
Note that the gradient $\nabla \ffun$ of the data-fidelity term is easily shown to be Lipschitz continuous. This either can be deduced from  the explicit expressions \eqref{eq:gradient2}, \eqref{eq:gradient3}
or by using $\nabla \ffun(\mu,\sigma) = \Fqpat' (\mu, \sigma)^* (\Fqpat \kl{\mu, \sigma} - v )$. In any case, the Lipschitz continuity of $(\mu,\sigma) \mapsto \nabla \ffun(\mu,\sigma)$
follows from similar arguments as in the proofs of the Lipschitz-continuity of $\Fo$ and $\Ho$,
given  in Section \ref{sec:F} and Section \ref{sec:H},  respectively.
\end{remark}

Convergence of the proximal gradient algorithm  \eqref{eq:proxgrad} is
well known for the case that $\ffun$ is convex with $\beta$-Lipschitz
continuous gradient and step sizes
satisfying $s_n\in [\eps, 2/\beta-\eps]$ for some constant $\eps >0$,
see \cite{ComWaj05,CombPes11}. These results are also valid for infinite dimensional
Hilbert spaces.  Because our forward operator $\Fqpat$ is nonlinear, the data-fidelity term $\ffun$
is non-convex and these results are not directly applicable to qPAT. Note however 
Recently, the converge analysis of the proximal gradient analysis has been extended to the case of
non-convex functionals; see  \cite{attouch2013convergence,BolSabTeb14,ChoPesRep14}.

\section{Numerical implementation}
\label{sec:num}

Our numerical simulations are carried out in $d=2$ spatial dimensions.
The  stationary RTE is solved on a  square domain $\Om={\blau [-1,1]^2}$.
For the scattering  kernel we choose the two dimensional version of the Henyey-Greenstein kernel,
\begin{equation*}
	k(\theta,\theta')
	\coloneqq
	\frac{1}{2\pi}\frac{1-g^2}{1+g^2-2g\cos(\theta\cdot\theta')}
	\quad \text{ for } \theta, \theta' \in \sph^1 \,,
\end{equation*}
where $g \in (0,1)$ is the  anisotropy factor.
Before we present results of our numerical simulations we first outline how we
numerically solve  the stationary RTE in two spatial dimensions. This step 
 is required for evaluating both, the forward operator $\Fqpat$ and the
 gradient $\nabla \ffun$ of the data-fidelity term.

\subsection{Numerical solution of the RTE}
\label{sec:FE}

For the numerical solution of the stationary RTE \eqref{eq:srt}, \eqref{eq:srt-b} we employ a finite element method.
For that purpose one calculates the weak form of equation \eqref{eq:srt}, \eqref{eq:srt-b} by integrating the equation against a test function $w \colon \Om\times \mathbb{S}^1 \to \R$. Integrating by parts in the transport term yields
\begin{equation}\label{eq:weak}
\int_{\Om}
\int_{\mathbb{S}^1}
\kl{-\theta\cdot\nabla_x w + \mu w+ \sigma w- \sigma  \Ko w}
  \Phi  \, \rmd \theta \, \rmd x
+\int_{\partial \Om \times \sph^1} \Phi w \kl{\theta\cdot \nu}
\rmd\sigma
=
\int_{\Om}
\int_{\mathbb{S}^1}
q w  \, \rmd \theta \, \rmd x  \,.
\end{equation}
Here we dropped all dependencies on the variables to shorten notation and $\rmd \sigma$ denotes the usual surface measure on  $\partial \Om \times \sph^1$.

The numerical scheme replaces the  exact solution
by a linear combination in the finite element space
\begin{equation} \label{eq:fe}
	\Phi^{(h)} (x,\theta)
	=
	\sum_{i =1}^{N_h} c_i^{(h)}  \psi_i^{(h)} (x,\theta) \,,
\end{equation}
where any $\psi_i^{(h)}(x,\theta)$ is the product of a basis function in space and a basis function in velocity and the sum ranges over all possible combinations.
The spatial domain is triangulated uniformly with mesh size $h$ as illustrated in Figure~\ref{fig:triangulation}.
The velocity direction on the circle is divided into $16$ equal subintervals.
We use $P_1$ Lagrangian elements, i.e. piecewise affine functions, in the two dimensional spatial domain as well as for the angle.

\psfrag{a}{$x_1$}
\psfrag{b}{$x_{ N+1}$}
\psfrag{c}{$x_{( N+1)^2}$ }
\begin{figure}[htb!]
\floatbox[{\capbeside\thisfloatsetup{capbesideposition={right,bottom},capbesidewidth=0.5\textwidth}}]{figure}[\FBwidth]
{\caption{\textsc{Spatial finite element discretization.} The square domain $\Om = {\blau [-1,1]^2}$ is divided  into $2 N^2$ triangles. To any of the  $(N+1)^2$ grid points $x_1, \dots, x_{(N+1)^2}$ a piecewise affine basis function is associated, that takes the value one at one  grid point and the value zero on all other grid points, and is affine on every triangle.}\label{fig:triangulation}}
{ \includegraphics[width=0.25\textwidth]{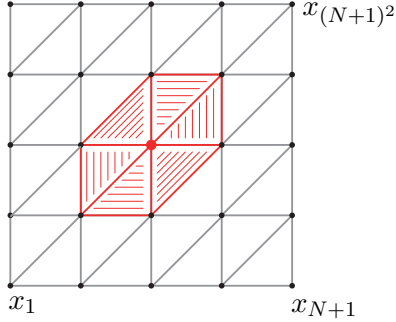}\hspace{0.15\textwidth}}
\end{figure}

To increase stability in low scattering areas we add some artificial diffusion in the transport direction.
This is called the streamline diffusion method, see for example \cite{Kan10} and the references therein. In the streamline diffusion method the  solution $\Phi$ is approximated in the usual way by  \eqref{eq:fe}. However, the
test functions take the form
\begin{equation} \label{eq:fe-sd}
	w(x,\theta) = \sum_{j =1}^{N_h} w_j\skl{\psi_j(x,\theta) +
	\delta(x,\theta)\, \theta\cdot
	\nabla_x\psi_j(x,\theta)} \,,
\end{equation}
where  the additional term introduces some artificial
diffusion.
In our experiments, the stabilization parameter is taken as
$\delta(x,\theta) = 3 h/100$ for  $\mu(x)+\sigma(x,\theta) + <1$
and zero otherwise. Note that the streamline diffusion method provides  a
fully consistent stabilization of the original problem.

Making the ansatz  \eqref{eq:fe} for the numerical solution and using  test functions of the form \eqref{eq:fe-sd}, equation \eqref{eq:weak} yields a system of linear equations $M^{(h)} c^{(h)} = b^{(h)}$ for the coefficient vector of the numerical solution.
The entries of  $M^{(h)}$ and   $b^{(h)}$ can be calculated by setting $\Phi = \psi_i$ and $w = \psi_j  + \delta \theta\cdot \nabla_x\psi_j $.  For simplicity  we only consider the case $q=0$ corresponding to zero sources of internal illumination. Then, similar  to \eqref{eq:weak} we obtain
\begin{multline}\label{eq:weak-sd}
 \int_{\Om}\int_{\mathbb{S}^1} \left(\delta \theta\cdot\nabla_x\psi_i-\psi_i\right)\theta\cdot\nabla_x\psi_j\rmd \theta\rmd x+\int_{\Gamma_+ }
 \abs{\theta\cdot \nu}\psi_i\psi_j\rmd\sigma \\
\int_{\Om}\int_{\mathbb{S}^1}(\mu+\sigma -  \sigma \Ko )\left(\psi_j+\delta\theta\cdot\nabla_x\psi_j\right)\psi_i\rmd \theta\rmd x
= \int_{\Gamma_-}
 \abs{\theta\cdot \nu} \psi_i\psi_j\rmd\sigma \,.
\end{multline}
The  entries of  the matrix $M^{(h)}$ now can be calculated by evaluating the
integrals on the left hand side of \eqref{eq:weak-sd}.
The right hand side of \eqref{eq:weak-sd} together with the prescribed boundary light sources on
$\Gamma_-$  yields the entries of $b^{(h)}$.

  \begin{figure}[htb!]
\begin{center}
  \includegraphics[width=0.3\textwidth]{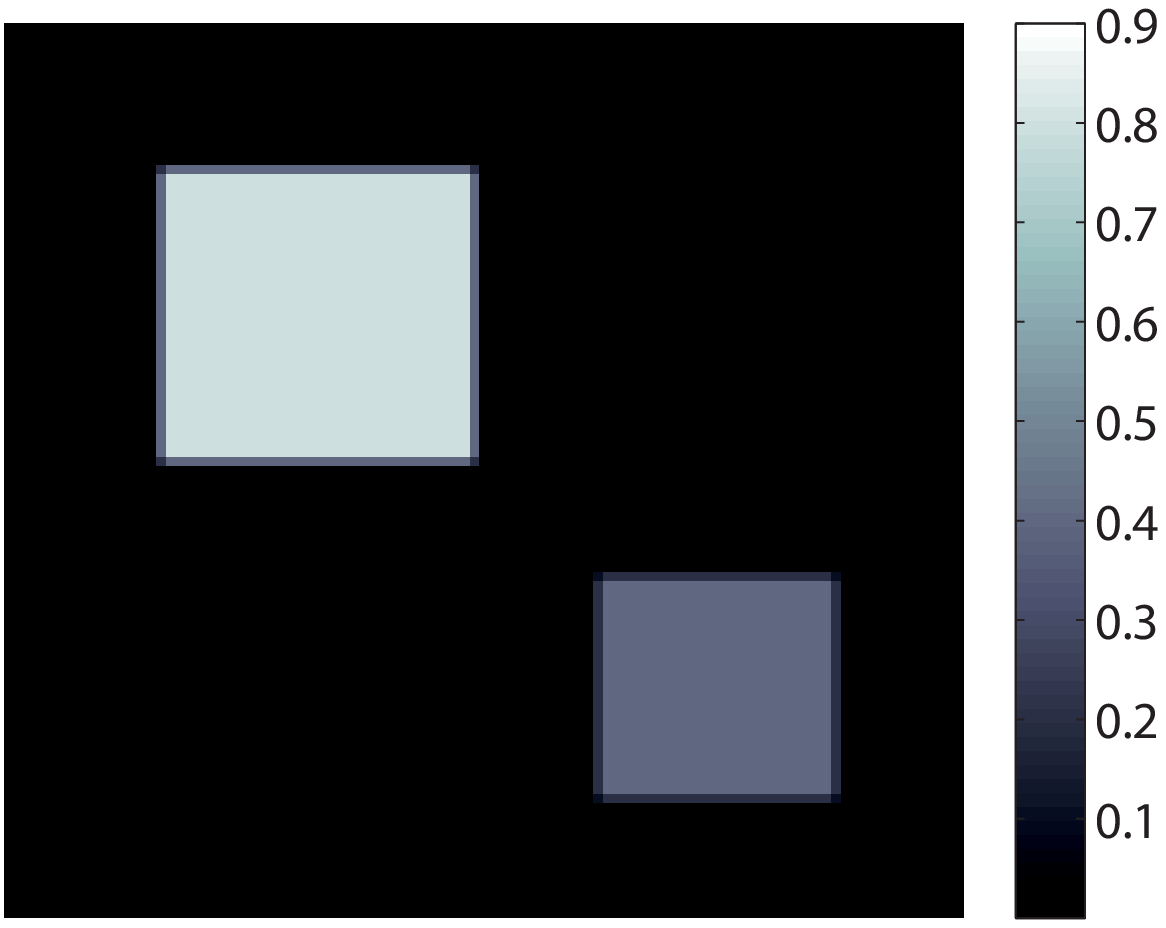}\quad
  \includegraphics[width=0.3\textwidth]{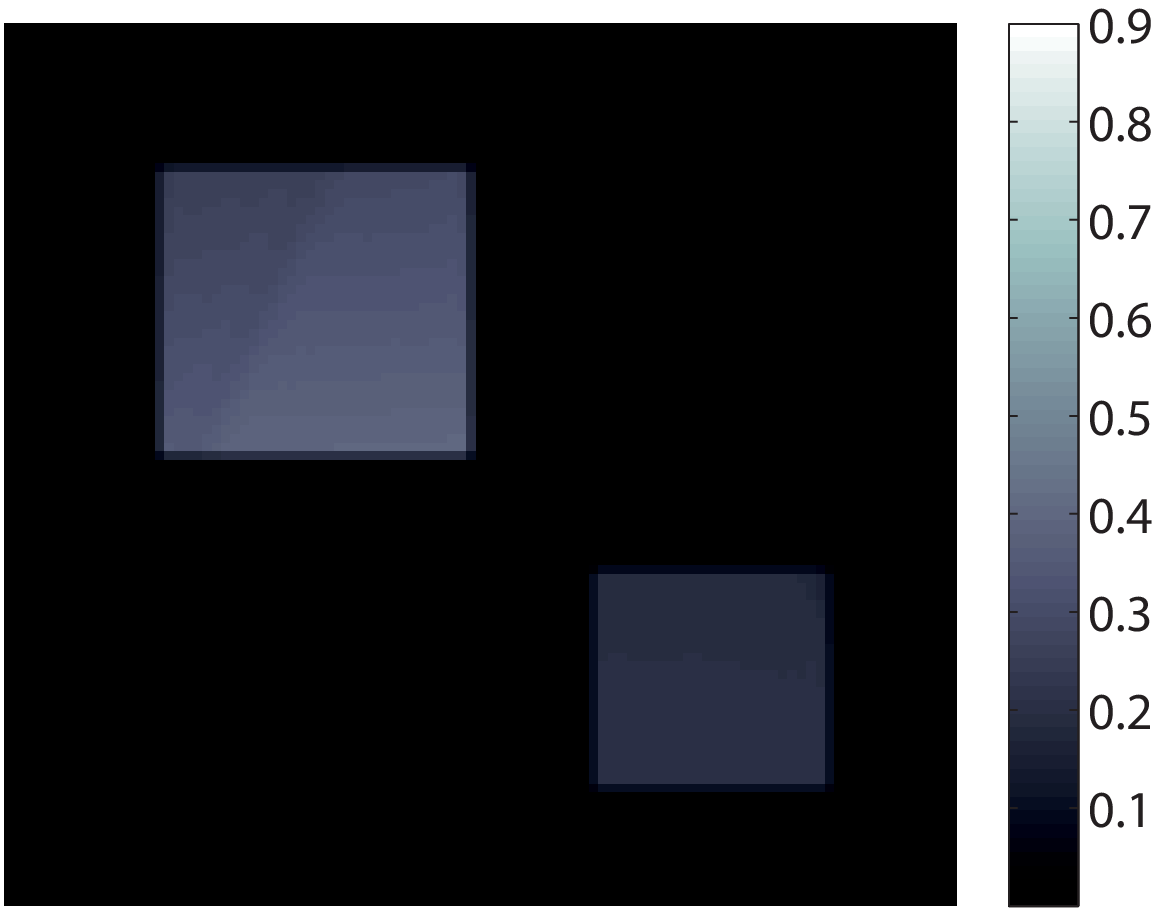}
  \quad
  \includegraphics[width=0.32\textwidth]{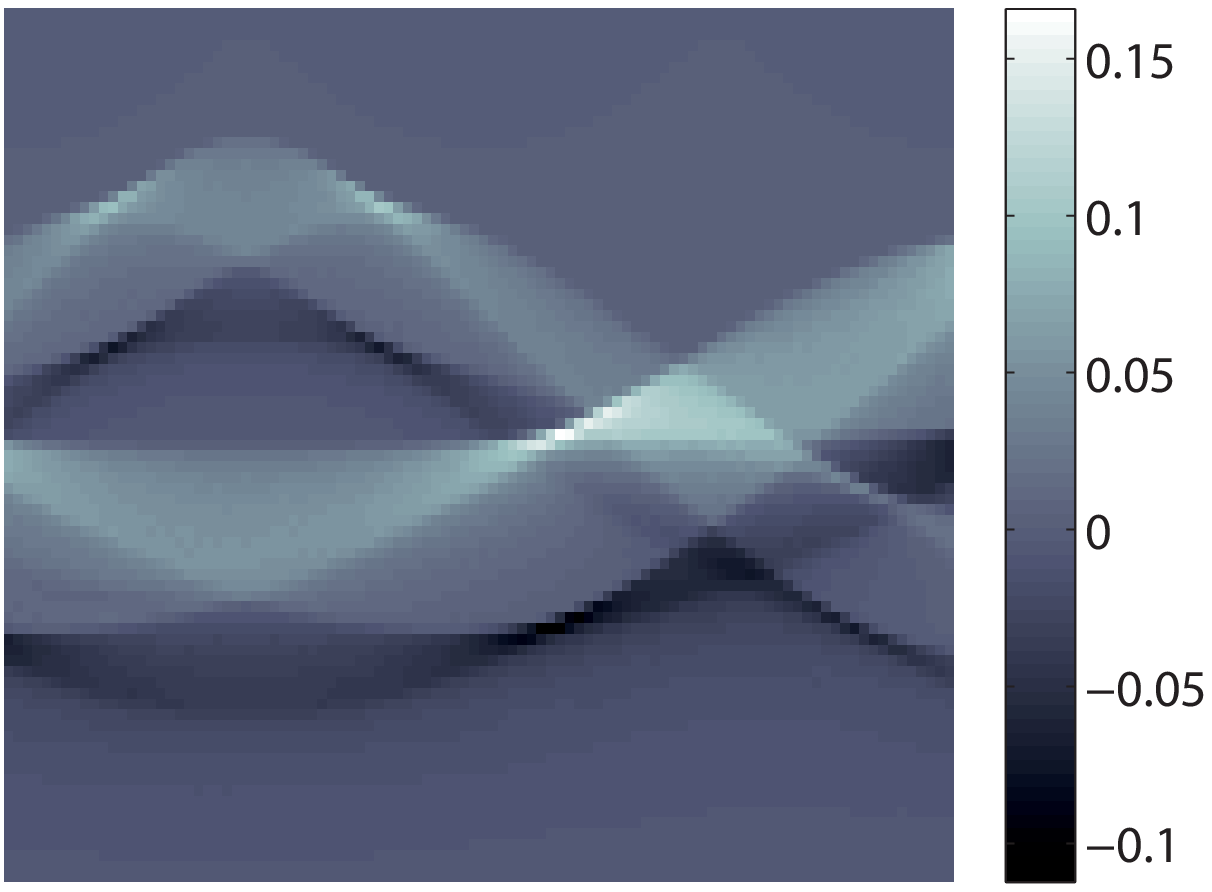}
\end{center}
\caption{\blau \textsc{Original phantom and simulated data.} True absorption coefficient (left),
simulated heating function (middle), and  simulated  pressure data on the boundary (right; the detector location varies in the horizontal and time in the vertical direction).}\label{fig:sim1}
\end{figure}

\subsection{Numerical  results}

For the  following  numerical results, the stationary RTE is solved by the finite element method outlined in  Subsection~\ref{sec:FE}.
For that purpose  the domain $\Om = {\blau [-1,1]^2}$ is discretized  by a mesh {\blau of}  triangular elements (compare Figure~\ref{fig:triangulation}).
The angular domain is divided into $16$ subintervals of equal length.
The  anisotropy factor is taken as $g = 0.6$ and the scattering coefficient is taken as  $\sigma = 3$.
We use a single {\blau boundary source distribution $f$} representing a planar illumination along the  lower edge ${\blau [-1,1] \times \set{-1}}$, {\blau where all photons enter the domain vertically}.
{\blau For solving the inverse problem we used  a mesh containing  $7442$ triangular elements. In order to avoid performing inverse crime, for simulating the data we used a finer mesh containing  $20402$ triangular elements.}

The solution of the two-dimensional wave equation  \eqref{eq:wave-ini}
is computed by numerically evaluating the solution formula \eqref{eq:sm}, where
the detection curve $\La =  \set{3/2 (\cos \varphi, \sin \varphi) : \varphi \in (-\pi, 0)} $ is a half-circle  on the boundary of $B_{3/2}(0)$. The adjoint $\Wo_{\Om,\La}^* h$ is evaluated by numerically implementing  \eqref{eq:wave-ad}.
This can be  done efficiently by a filtered backprojection algorithm as described in \cite{BurBauGruHalPal07,FinHalRak07}.
The geometry of $\Om$ and $\La$ is similar to the one illustrated in Figure~\ref{fig:geometry}.

   \begin{figure}[htb!]
\begin{center}
  \includegraphics[width=0.3\textwidth]{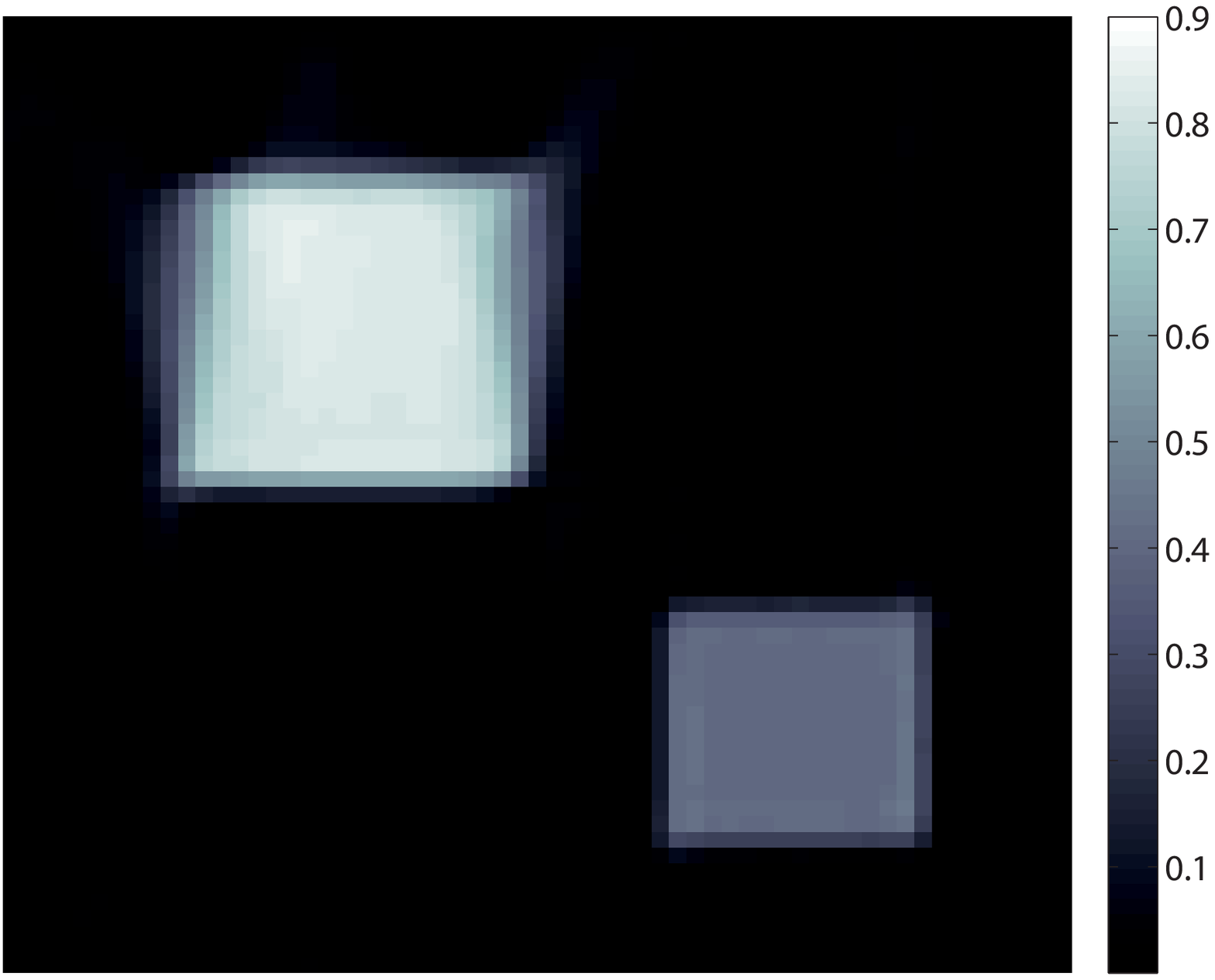}
  \quad
  \includegraphics[width=0.3\textwidth]{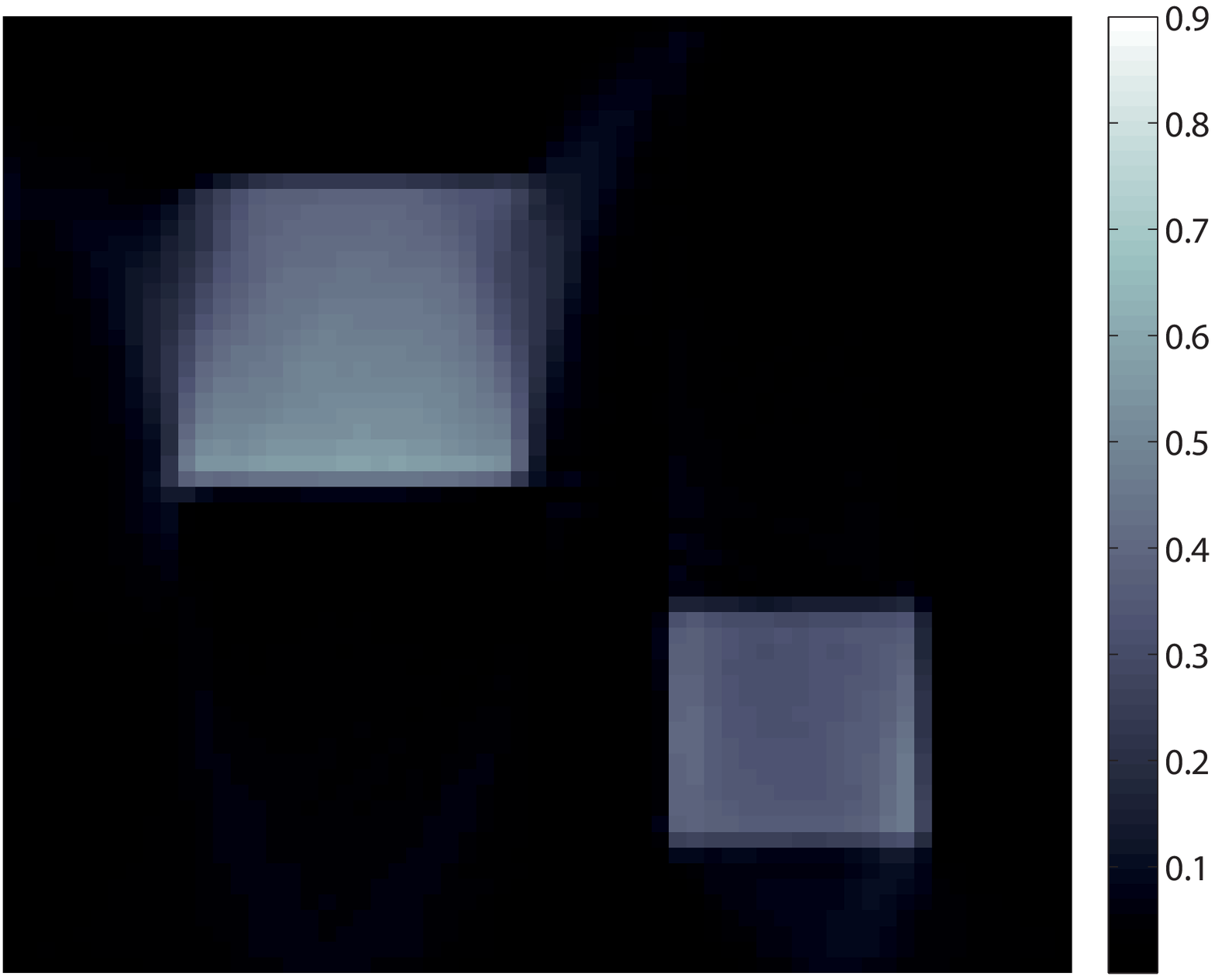}
\end{center}
\caption{\textsc{Reconstruction  results for simulated data.}
{\blau Reconstructed absorption coefficient using our single-stage approach (left),}  and  reconstructed absorption coefficient using the usual two-stage approach (right).}\label{fig:sim}
\end{figure}

For our initial experiments we assume  the  scattering coefficient $\sigma$ to be known. In such a situation, the  proximal gradient algorithm  outlined in the Subsection~\ref{sec:min} reads
 \begin{equation*}
\mu_{n+1}
\coloneqq
\prox_{s_n\lambda\rfun}  \kl{ \mu_n
-  s_n  \nabla_\mu \ffun \kl{\mu_n, \sigma}} \quad \text{ for } n \in \N
\end{equation*}
where $s_n>0$ is the step size,
$\nabla_\mu \ffun(\mu,\sigma)$
 is the gradient of $\ffun$ in the first component, given by \eqref{eq:gradient2},
 and $\prox_{s_n\lambda\rfun} \kl{\,\cdot\, }$
is the proximity operator similar as in \eqref{eq:prox}.
In the presented numerical examples the regularization term is taken as a quadratic functional $\rfun(\mu) = \tfrac{1}{2} \norm{\partial_x \mu}_{L^2(\Om)}^2 + \tfrac{1}{2} \norm{\partial_y \mu}_{L^2(\Om)}^2$. In order to speed  up the iterative scheme we compute the proximity operator only approximately by projecting the unconstrained minimizer
 $\argmin_{\mu} \tfrac{1}{2} \snorm{\mu- \hat\mu}^2 + s_n \lambda\rfun\skl{\mu}$ on $\dom_2$. Therefore the  main numerical cost in the proximal step is the solution of a linear equation, which is  relatively cheap compared to the evaluation of  the gradient $\nabla_\mu \ffun$.

In {\blau Figures~\ref{fig:sim1} and~\ref{fig:sim}} we present results of our numerical {\blau experiments}.
{\blau Figure~\ref{fig:sim1} shows the true
absorption coefficient as well as the heating function and the simulated pressure data.  The left image in Figure~\ref{fig:sim} shows the}
numerical reconstruction with the proposed single-stage
approach using 40 iterations of the proximal gradient algorithm.
{\blau We observed empirically, that the proximal gradient algorithm
stagnated after 20 to 40 iterations  and therefore we used 40 as a stopping index.}
For comparison purpose, the right image in Figure~\ref{fig:sim}
shows reconstruction results using the classical two-stage approach.
For that purpose we apply Tikhonov regularization and the proximal gradient  algorithm to the inverse problem
$h =  \Ho_i \skl{\mu} + z_h $. {\blau Again the iteration has been stopped after 40 iteration, where the iteration has been stagnated.}
The approximate heating $h$ is computed numerically by applying the two dimensional universal backprojection formula
 \cite{BurBauGruHalPal07,Kun07a,Hal13a}
 to the acoustic data $v = \Wo_{\Om, \La} \circ \Ho_i \skl{\mu} + z$.  All computations have been performed in  \textsc{Matlab}
 on a notebook with $\unit[2.3]{GHz}$
 Intel Core i7 processor. The total computation times have been 26 minutes for the  two-stage approach and 38 minutes for the single-stage approach.

One notices that in both reconstructions some  boundaries in the upper half are blurred. Such artifacts are  expected and arise from the ill-posedness of the acoustical problem when using limited-angle data; see \cite{FriQui14,paltaufnhb07,XuWanAmbKuc04}.
However  these artifacts are less severe for the single-stage algorithm than for
the classical two-stage algorithm. Further, in this example, the single-stage algorithm also yields a better quantitative estimation of the values of the absorption coefficient.

\begin{figure}[htb!]
\begin{center}
  \includegraphics[width=0.31\textwidth]{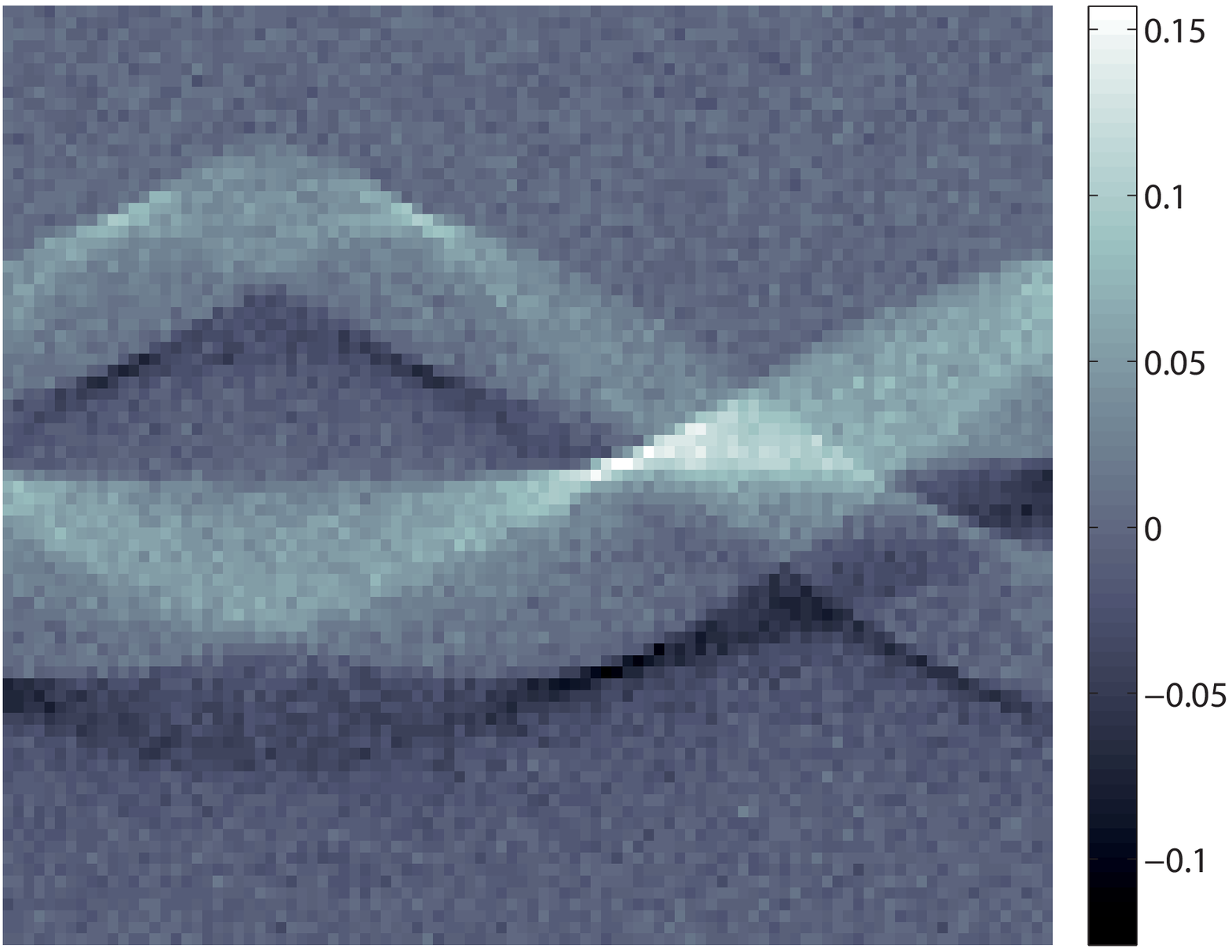}\quad
  \includegraphics[width=0.3\textwidth]{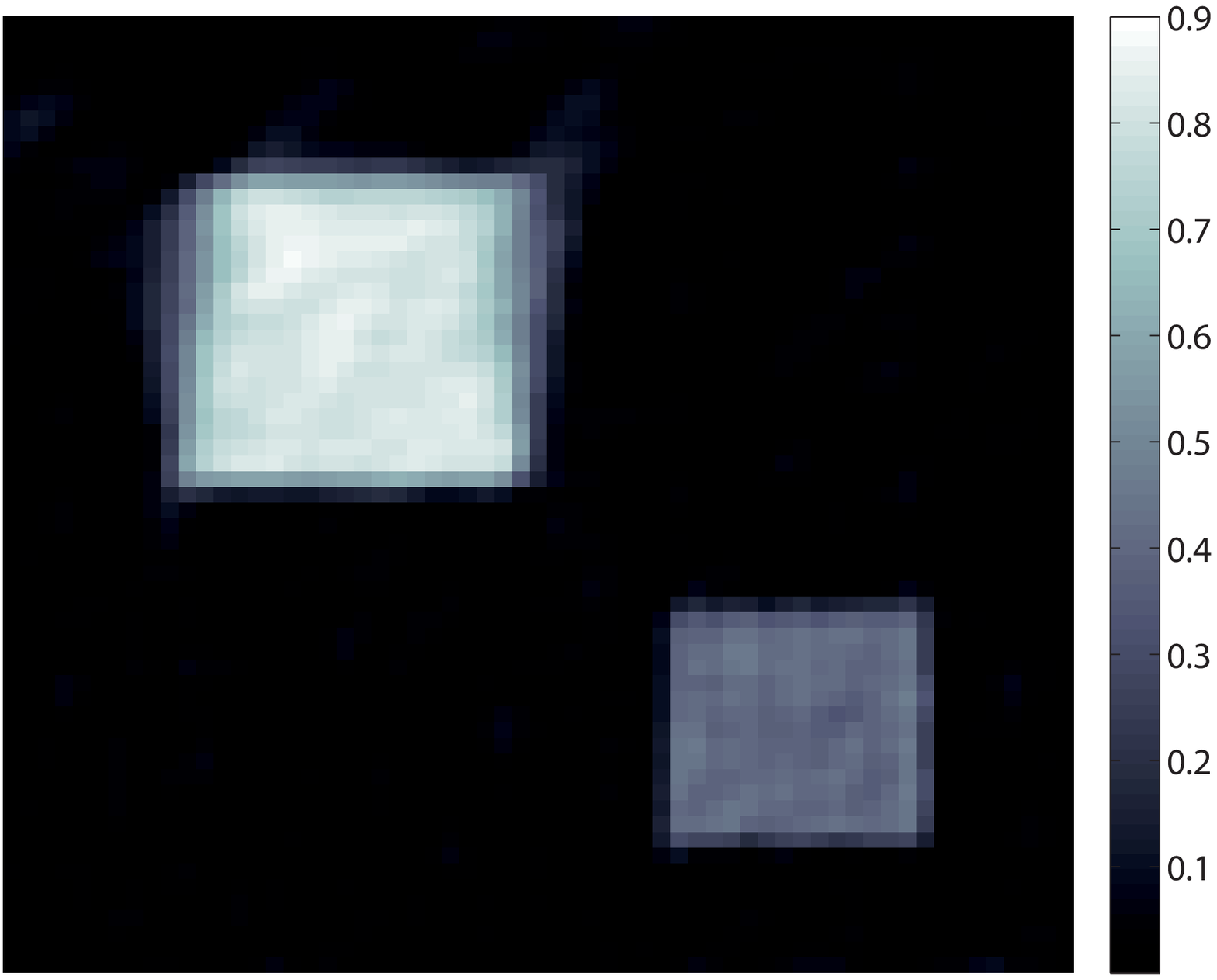}
  \quad
  \includegraphics[width=0.3\textwidth]{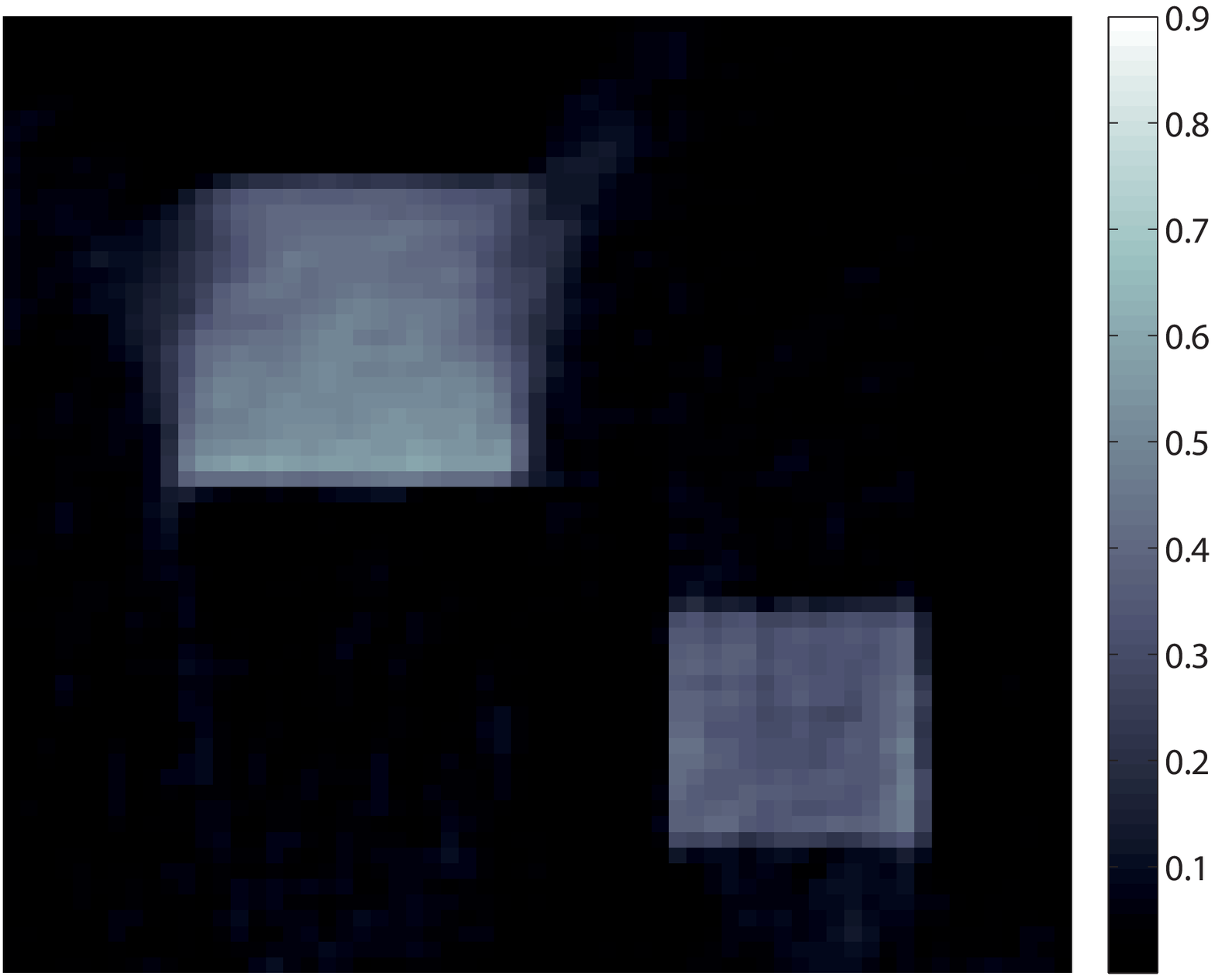}
\end{center}
\caption{\textsc{Reconstruction  results for noisy pressure data.}
Pressure data with $5\%$  added noise (left; {\blau detector location varies
in the horizontal and time in the vertical direction}), reconstructed absorption coefficient using our single-stage approach (middle)  the classical two-stage approach (right).}\label{fig:sim-noisy}
\end{figure}

Finally, in order to investigate the stability of the derived algorithms  with respect to noise, we applied
the single-stage and the two-stage algorithm  after adding Gaussian white noise to the data
with standard deviation equal to $5\%$ of the maximal absolute data values. 
{\blau Note that for both, the single-stage and the two-stage algorithm, noise has only been added
to the acoustic data.} 
The reconstruction results for noisy data are shown in Figure~\ref{fig:sim-noisy}.
As can be seen both algorithms are quite stable with respect to data perturbations. However, again, the single-stage approach yields better results and less artifacts than the two-stage algorithm.

\section{Conclusion}
\label{sec:conclusion}

In this paper we proposed a single-stage approach for quantitative PAT. For that purpose we derive algorithms that directly recover the optical parameters from the measured acoustical data. This is in contrast to the usual two-stage approach, where the absorbed energy distribution is estimated in a first step, and the optical parameters are reconstructed from the
estimated energy distribution in a second step.
Our single-stage algorithm is based on generalized Tikhonov regularization and minimization of the Tikhonov functional by the proximal gradient algorithm. In order to show that Tikhonov regularization is well-posed and convergent we analyzed the stationary radiative transfer equation \eqref{eq:srt}, \eqref{eq:srt-b} in a functional analytic framework.
For that purpose we used recent results of  \cite{EggSch14} that guarantees the well-posedness even in the case of voids.

We presented results of our initial numerical studies using a simple limited angle scenario, where the scattering coefficient is assumed to be known. In this situation our  single-stage algorithms has led to less artifacts than the two-stage procedure. More detailed numerical studies will be presented in future work. In that context, we will also investigate the use of multiple illuminations and multiple wavelength, which allows to also reconstruct the in general unknown scattering coefficient and Gr\"uneissen parameter.
We further plan to investigate the use of  more general regularization functionals such as the total variation in combination with the single-stage approach.

\section*{Acknowledgements}
This work has been supported  by the Tyrolean Science Fund (Tiroler Wissenschaftsfonds), project number 153722.
{\blau We want to thank the referees for useful comments which helped to improve our  manuscript,
and for bringing  reference \cite{Song:14} to our attention.}

\end{document}